\pdfoutput=1

\documentclass[11pt]{amsart}
\usepackage{fullpage }   

\usepackage[all,cmtip]{xy} \SilentMatrices

\usepackage{lscape}
\usepackage{amsmath}
\usepackage{amsfonts}
\usepackage{amssymb}
\usepackage{amsthm}
\usepackage{subfigure}
\usepackage{graphicx}
\usepackage{color}
\usepackage{enumerate}
\usepackage{lineno}
\usepackage{pinlabel}
\usepackage{dcolumn}
\newcolumntype{2}{D{.}{}{2.0}}

\usepackage{comment}

\newtheorem{theorem}{Theorem}[section]
\newtheorem{proposition}[theorem]{Proposition}
\newtheorem{lemma}[theorem]{Lemma}

\theoremstyle{definition}
\newtheorem{definition}[theorem]{Definition}

\newtheorem{example}[theorem]{Example}

\theoremstyle{remark}
\newtheorem{remark}[theorem]{Remark}

\numberwithin{equation}{section}

\newcommand{\bal}{\hat{\alpha}}

\newcommand{\bga}{\hat{\gamma}}

\newcommand{\al}{ \alpha}

\newcommand{\ga}{ \gamma}

\newcommand{\g}{\xi}

\newcommand{\fS}{\mathfrak{S}}

\newcommand{\calG}{\mathcal{G}}

\newcommand{\inv}{^{-1}}

\newcommand{\mycase}[1]{\medskip\medskip\noindent$\bullet${\it \, Case #1 }}

\newcommand{\Aut}{\mbox{Aut}\,}
\newcommand{\D}{\mathcal{D}}

\begin{document}

\title{New Dualities From Old: generating geometric, Petrie, and Wilson dualities and trialities of ribbon graphs}

\author[L.~Abrams]{Lowell Abrams}
\address{	University Writing Program and Department of Mathematics,
George Washington University,
2115 G Street NW, Washington, DC  20052, USA}
\email{labrams@gwu.edu}

\author[J.~Ellis-Monaghan]{Joanna A. Ellis-Monaghan}
\address{Department of Mathematics, Saint Michael's College, 1 Winooski Park, Colchester, VT 05439, USA.  }
\email{jellis-monaghan@smcvt.edu}
\thanks{The work of the second  author was supported by the National Science Foundation (NSF) under grants  DMS-1001408 and EFRI-1332411.
}

\date{\today}

\begin{abstract}

We develop an algebraic framework for ribbon graphs that reveals symmetry properties of (partial) twisted duality.  The original ribbon group action of Ellis-Monaghan and Moffatt restricts self-duality, self-petriality, or self-triality to the canonical identification between the edges of a graph and those of its dual, petrial, or trial, whereas the more natural usual definition allows \emph{any} isomorphism. Here we define a new ribbon group action on ribbon graphs that uses a semidirect product of the original ribbon group with a permutation group to take (partial) twists and duals of ribbon graphs while simultaneously encoding graph isomorphisms.  This brings new algebraic tools to bear on the natural definitions of self-duality etc., as a ribbon graph is a fixed point of this new ribbon group action if and only if it is isomorphic to one of its (partial) twisted duals. Using these new tools, we show that every ribbon graph has in its orbit an orientable embedded bouquet and prove that the (partial) twisted duality properties of these bouquets propagate through their orbits. Thus, all (partial) twisted duality properties of general embedded graphs, especially those of being self-dual, self-petrial, and self-trial, may be analyzed through orientable embedded bouquets, for which checking isomorphism reduces simply to checking dihedral group symmetries. Previous research on self-duality, self-petriality, and self-triality typically focused on highly symmetric regular maps. However, the theory here fully encompasses all cellularly embedded graphs.  In contrast with the few, large, very high-genus, self-trial regular maps found by Wilson, and by Jones and Poultin,  here we apply the new ribbon group action to generate \emph{all} self-trial ribbon graphs on up to seven edges. We also show how the automorphism group of a graph may be used to find self-dual, -petrial or –trial graphs in its orbit, thus exposing the relationship between regularity and the ribbon group action. This strategy yields an infinite family of self-trial graphs on $3m$ edges, for all $m$, that do not arise as covers or parallel connections of regular maps, thus answering a question of Jones and Poulton.

\end{abstract}

\maketitle

\section{Introduction}

Surface duality and Petrie duality of cellularly embedded graphs have a long history, for both plane graphs and those embedded in other surfaces, as finding and characterizing self-dual and self-Petrial graphs are central to the study of graph symmetry.  In \cite{Wil79}, Wilson examined the action on regular maps generated by surface duality and Petrie duality; these are operations of order two that do not commute, and hence yield an action of $S_3$ on regular maps, with the action of the elements of order three called \emph{triality}. Of primary interest of course are graphs that are self-dual, self-Petrial, self-trial, or all of these.  Subsequently,  surface and Petrie dualites were refined to the individual edges of general embedded graphs, first with Chmutov's partial duality in \cite{Chm09}, and then the twisted duals of \cite{EMM12, EMM13a}. We refer to these full and partial dualities and trialities that arise from twisted duality in aggregate as \emph{twuals}, speaking of self-twualities, twualizing a graph, etc., thus coining a word that captures a sense of both `twist' and `dual', as well as  `duality' and `triality', while having the necessary grammatical forms analogous to those of the word `dual'\footnote{An aesthetically pleasing word that encompasses the six kinds of twisted duality and their partial applications with all the grammatical forms of `duality' proves to be elusive. Acronyms such as PTD (Partial Twisted Dual) don't have analogs of  \emph{e.g.} `dualize' or `duality', while other terms such as `mutable' or `flippable' either have been claimed for other settings or have no linguistic resonance with twisted duality.}.

Here we situate the various forms of graph twuality in a new, more finely grained, setting, presenting a new algebraic framework for determining various forms of self-twuality for cellularly embedded graphs, including not only the surface duals, Petrie duals, and trials (all the direct derivatives of Wilson \cite{Wil79}) but also any application of partial duals and twists.  Critically,  this new ribbon group action encompasses the more common \emph{natural} twuality, which allows any isomorphism between a graph and its twual, and not only the \emph{canonical} twuality of prior ribbon group actions, where twuality was restricted to the canonical identification of edges between a graph and its twual.   We then leverage this framework to show that all forms of self-twuality  may be completely captured through the orbits of orientable embedded bouquets (OEBs), \emph{i.e.} single vertex orientable graphs. 

The ribbon group action defined here provides theoretical tools applicable to the many current directions in partial twuality. In one direction there has been interest in the forms of partial duals of general graphs, particularly their genus(es)  as well as when the partial duals are Eulerian or bipartite. For instance, Ellingham and Zha investigate when a partial dual has the property that every face is bounded by a cycle \cite{EZ17}, Huggett and Moffat characterize when partial duals of plane graphs are bipartite \cite{HM13}, and Metsidik and Jin characterize when partial duals of plane graphs are Eulerian \cite{MJ18}. In another direction, research focuses on full self-twuals for regular maps and when they may be self-dual, self-petrial, or self-Wilsonial, \cite{Cond09, JP10, RSW}. In some sense halfway between these areas is Orbani\'c {\it et al}, who use Wilson's operations to generate examples of $k$-orbit maps, a slight loosening of regularity \cite{OPW}. We give examples illustrating how the algebraic framework developed here can further these directions.

In particular, to illustrate the power of the techniques we are introducing here, we show how manipulating OEBs with the ribbon group action leads to a systematic way of generating graphs with any desired form of self-twuality, for all ribbon graphs and not just regular maps. The techniques here also apply to discovering graphs with desired self-twuality properties in the orbits of highly symmetric graphs, which have rich automorphism groups.  Indeed, regular maps are a prime example of such graphs, which illuminates why regular maps have been a natural search space in the past for self-twual maps of various kinds.

Previous studies such as \cite{JP10, RSW, SS95, SS96, Wil79} of self-dual, self-Petrial, and self-trial graphs have largely focused on either plane graphs or regular maps. One benefit of the approach presented here is that it provides a method of generating and studying self-dual, self-Petrial, and in particular self-trial, graphs in arbitrary surfaces with no assumption of any form of regularity at all.  For example, it was originally conjectured by Wilson  that there are no Class III regular maps, that is regular maps that are self-trial without also being self-dual or self-Petrial.  Wilson \cite{Wil79} subsequently found an non-orientable dual pair of type $\{9,9\}_9$ on 126 edges with characteristic $-70$; Jones and Poulton \cite{JP10}, leveraging a computer search of Conder (referenced in \cite{JP10}), proved that these have the greatest possible Euler characteristic, while the lowest possible genus in the orientable case, corresponding to a regular map of type $\{16, 16\}_{16}$, is 193. They also give constructions such as coverings and parallel connections that yield some infinite families of Class III regular maps, but ask for families that do not arise in this way.  Here, we produce such examples by
reducing the search space to OEBs.  This facilitates an exhaustive computer search for Class III examples among general graphs, and we find that there are in fact a large number of small, low genus, self-trial graphs that are not also self-dual and self-Petrial.  We also find an infinite family of Class III graphs that do not arise as covers or parallel connections of regular maps, thus answering the question posed by Jones and Poulton.

At the heart of our constructions is the very simple classical result that if a graph $G$ is self-Petrial, that is if $G^{\times} = G$, then $H=(G^*)^{\times}$ is self-dual.  In particular, new self-dualities arise from old through conjugation in Wilson's group action.  Here we solve for special kinds of conjugations, or `near-conjugations', within the ribbon group action to discover new self-twual graphs in the orbit of a given graph with some form of self-twuality, thus generating new highly structured graphs. To do this, however, we must first devise a manageable search for the initial graphs with known self-twuality properties.  OEBs serve in this role, since they  can be encoded by chord diagrams and hence can be systematically generated.  Furthermore, checking for isomorphism between OEBs involves considering only the dihedral group action on the chord diagrams, a much more manageable process than general graph isomorphism.

In Sections \ref{sec.Ribbon Graphs} and \ref{sec.wilson group} we review the basics of ribbon graphs and (partial) twuality. We then define our main algebraic object in Section \ref{sec.Ribbon Group}; the key idea is that we fully develop the ribbon group construction of \cite{EMM12, EMM13a} by labeling edges and incorporating edge-permutations into the group action.  With this, the algebraic framework then encompasses the most natural understanding of self-twuality-- e.g. that a graph is considered self-dual if there exists \emph{any} isomorphism between itself and its dual.  This is in contrast to the original ribbon group action given by \cite{EMM12, EMM13a} which restricts self-twuality to the canonical bijection of edges between a graph and its twual.  Definitions \ref{gammadef} and \ref{uniformdef} identify within this framework the various notions of self-twuality underpinning our main applications. From our algebraic perspective, the study of self-twuality becomes the study of stabilizers of the group action.

Our main results, Theorems \ref{canonical equivs}, \ref{natural prop}, and \ref{natural equivs},  show that self-twuality, in its various forms, propagates throughout the orbits of the ribbon group action. 
In Section \ref{sec:OEB} we prove that that every orbit contains an OEB, which provides essential leverage for the examples of Section \ref{Search results}. In particular Figure \ref{fig:selfies} lists, up to isomorphism, all self-trial non-self-dual graphs on up to 7 edges, and Proposition \ref{prop:infinite examples} shows how Theorem \ref{automorph duals} can be used to generate an infinite class of self-trial non-self-dual graphs.

After discussing in Section \ref{Code discussion} the implementation of the computation reported in Figure \ref{fig:selfies}, Section \ref{sec:further directions} closes the paper with a sample of the many further research directions arising from this algebraic framework for exploring graph twualities.

\section{Ribbon graphs and arrow presentations}\label{sec.Ribbon Graphs}

We first give a very brief reminder of some of the various ways to represent a graph embedded in a surface.  A more detailed treatment may be found in \cite{EMM13a}. 

We begin with \emph{ribbon graphs} as these will be the central objects of this paper.  

A \emph{ribbon graph} $G = (V(G),E(G))$ is a surface with boundary presented as the union of two sets of discs, a set $V(G)$ of \emph{vertices} and a set $E(G)$ of \emph{edges}, satisfying the following conditions:
\begin{enumerate}
    \item The vertices and edges intersect in disjoint line segments.
    \item Each such line segment lies on the boundary of precisely one vertex and precisely one edge.
    \item Every edge contains exactly two such line segments.
\end{enumerate}

A \emph{cellular embedding} of a graph $G$ on a closed compact surface $\Sigma$ is a drawing of $G$ on $\Sigma$ such that edges intersect only at their endpoints and each component of $\Sigma - G$ is homeomorphic to a disc. Two cellularly embedded graphs in the same surface are considered equivalent if there is a homeomorphism of the surface taking one to the other, and two cellularly embedded graphs are isomorphic if there is a homeomorphism of the surfaces that induces a graph isomorphism when restricted to the graphs embedded in the surfaces.  

A cellularly embedded graph can be obtained from a ribbon graph by gluing discs into the boundary components of the ribbon graph and and then retracting the ribbon graph in the resulting surface to get the embedding of the graph in the surface. See Figure \ref{figure:K4a}. Two ribbon graphs are isomorphic if they are isomorphic as cellularly embedded graphs.  

An arrow presentation is a concise way of representing a ribbon graph, due to Chmutov.  We simply draw the vertices of the ribbon graph as circles, and make small directed arcs where the edges intersect the vertices (we do not draw the edges).  The two arcs for a given edge are both directed clockwise (or equivalently both directed counterclockwise) if the edge has no twist, and directed as one clockwise and one counterclockwise if the edge is twisted. See Figure \ref{figure:K4a}.

\begin{figure} 
{\includegraphics[height=35mm]{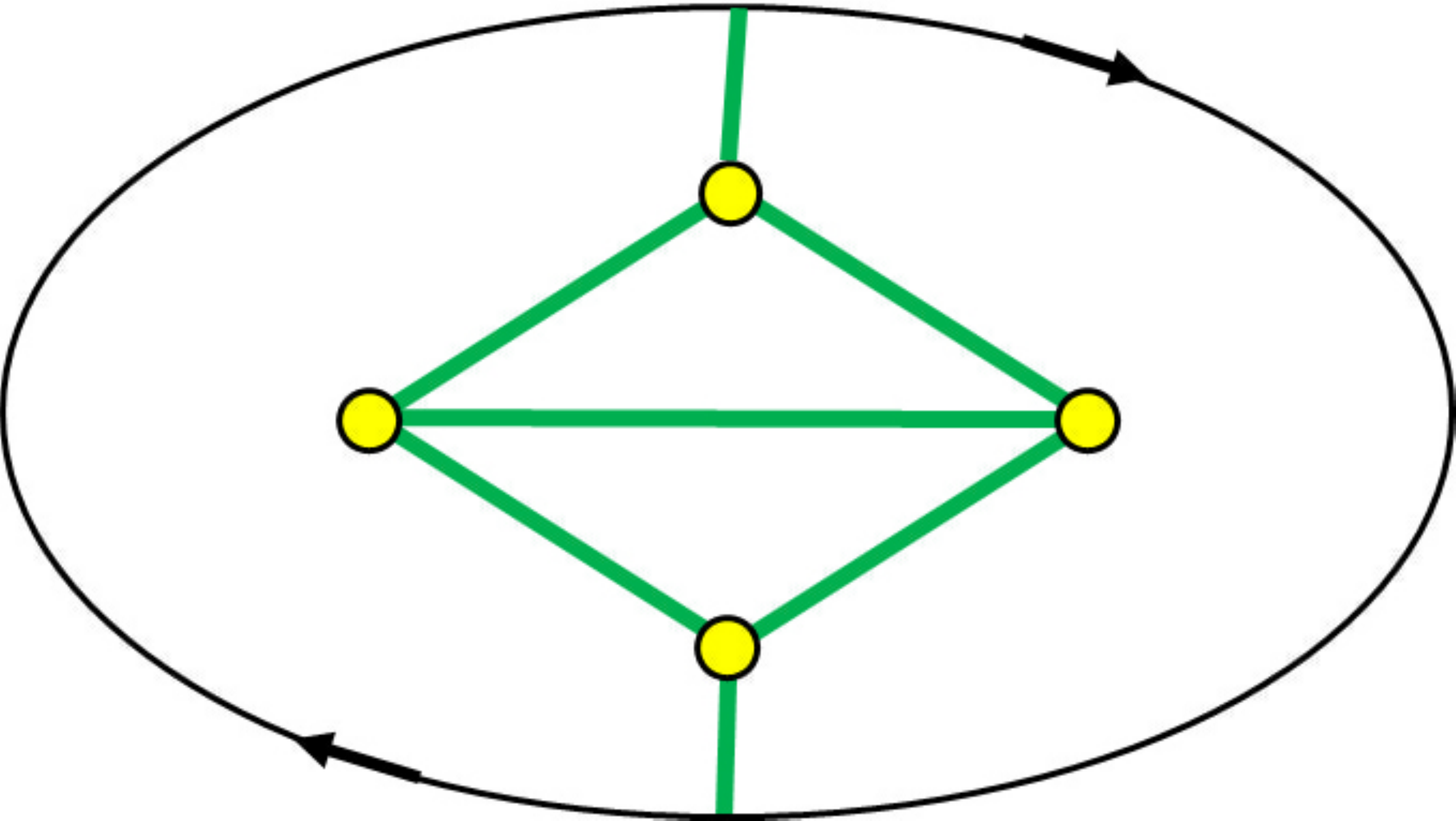}} 
\hspace{1cm} {\includegraphics[height=35mm]{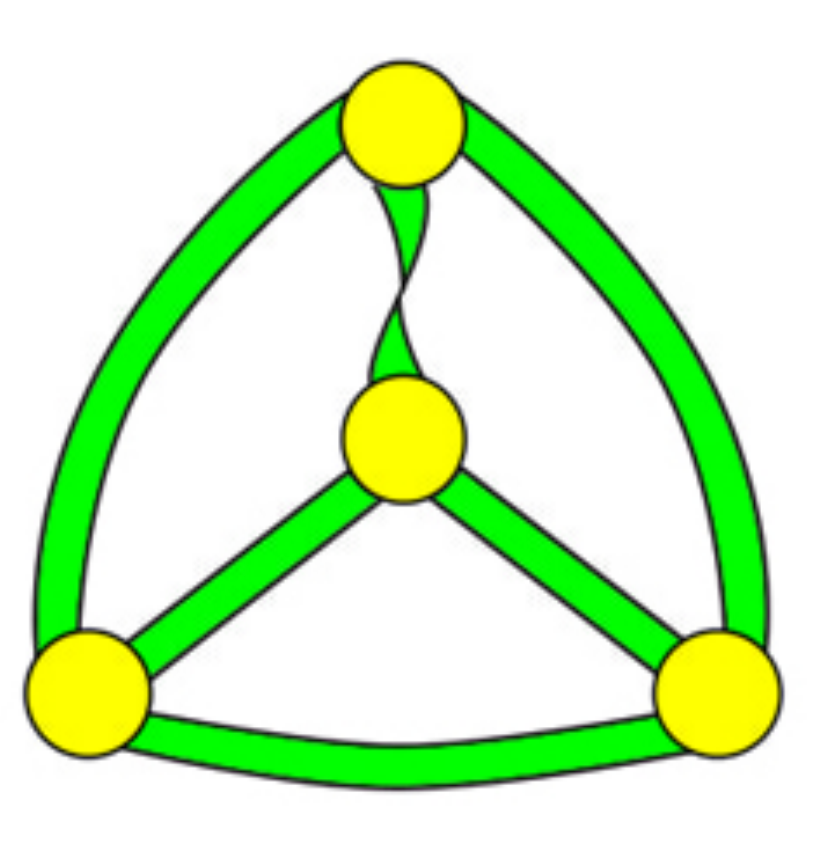}}
\hspace{1cm}
{\includegraphics[height=35mm]{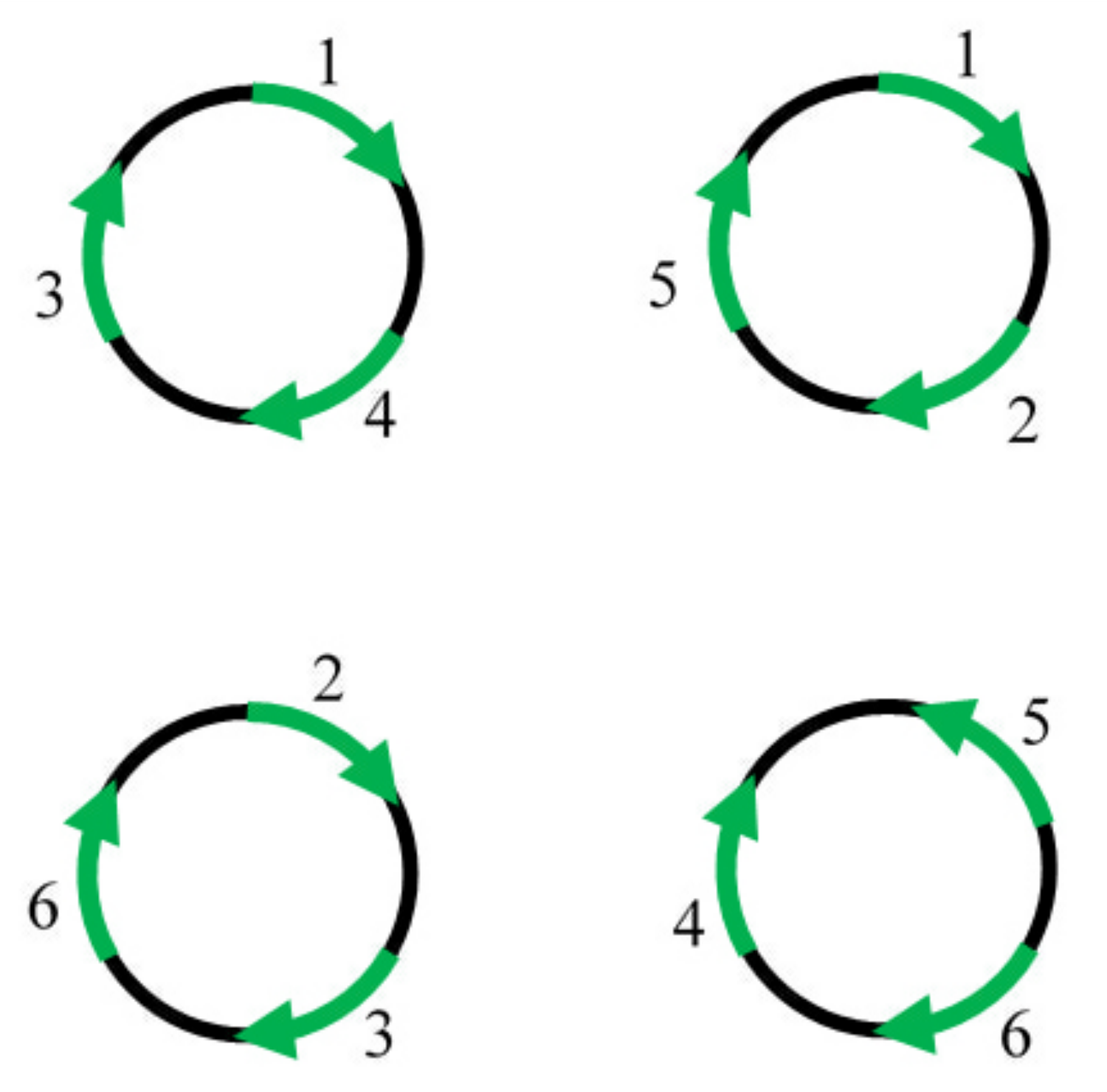}} 
\caption{Three ways of representing $K_4$ on the projective plane. }
\label{figure:K4a}
\end{figure}

Because we allow loops and multiple edges, and will use edge ordering to encode isomorphisms in the following sections, we will need the following formal definition of graph and graph isomorphism.  An abstract graph consists of a set $V$ of vertices and a set $E$ of edges, together with an incidence map $\psi$ taking $E$ to unordered pairs of elements of $V$ (elements may be repeated in the pair in the case of loops).    Two abstract graphs $G = (V_G, E_G)$ and $H=(V_H, E_H)$ are isomorphic if there are bijections $f\colon V_G \rightarrow V_H$ and $g\colon E_G \rightarrow E_H$ such that $\psi_G(e) = (u,v) \iff \psi_H(g(e)) = (f(u), f(v))$.

Since ribbon graphs are the relevant objects in the remainder of the paper, we simply use the word graph to indicate a ribbon graph.  We will use the term \emph{abstract graph} for the usual notion of a graph defined only in terms vertices and edges, and their incidences.

\section{Dualities and the Wilson group action}\label{sec.wilson group}

Two operations on an embedded graph, forming its geometric dual and forming its Petrie dual, are the foundation of the ribbon group action central to this paper. Furthermore, determining graphs that exhibit various forms of self-duality, including for geometric duals and Petrie duals, is the goal of this work.

The geometric dual of a cellularly embedded graph is formed by placing a vertex in the center of each face, and drawing an edge in the surface between two of these vertices whenever there is an edge shared by the faces that contain them. The geometric dual is contained in the same surface as the original graph. The Petrie dual of a graph has the same vertices and edges as the original graph, but the faces are given by `left-right' paths, achieved by choosing an edge, following it to one of its endpoints, and then turning either left or right to following one of its incident faces in the original embedding, and continuing this way, alternating the choice of turning left or right until the path closes. This process of building faces is repeated until every edge has been traversed exactly twice.  The Petrie dual is not generally embedded in the same surface as the original graph.    

Since we will be working primarily in the setting of ribbon graphs, we recall the formation of geometric and Petrie duals in that setting.  To form the geometric dual of a ribbon graph,  we sew discs into the boundary components, and remove the original vertex discs.  The new discs become the vertices of the dual, and the edges remain the same, although the intervals along which they coincide with the new vertices are the complements of those that coincided with the vertices of the original graph. The Petrie dual is formed for a ribbon graph by detaching one end of each edge from its incident vertex disc, giving the edge a half-twist, and reattaching it to the vertex disc.  

Chmutov's partial duality, and the partial Petrie duality given in \cite{EMM12}, apply these notions of duality to one edge at a time.
If $G$ is  an embedded graph given by an arrow presentation and $e$ is an edge of $G$, then the partial dual with respect to $e$ changes the arrow presentation
as follows. Suppose $A$ and $B$ are the two arrows labelled $e$ in an arrow presentation of $G$. Draw a line segment with an arrow on it directed from the the head of $A$ to the tail of $B$, and a line segment with an
arrow on it directed from the head of $B$ to the tail of $A$. Label both of these arrows $e$, then delete $A$ and $B$
and the arcs containing them to complete taking the partial dual with respect to $e$. To take the partial Petrial with respect to $e$, simply reverse the direction of the arrow on exactly one of the arrows labeled $e$. See Figure \ref{partials}.  A full orbit on a single edge is shown in Figure \ref{figure:EdgeOrbit}. Note that applying the partial dual operation to all edges gives the geometric dual, and applying the partial Petrial operation to all edges gives the Petrie dual. 

\begin{figure}
\[  \tau\left( \;\; \raisebox{-4mm}{\includegraphics[height=10mm]{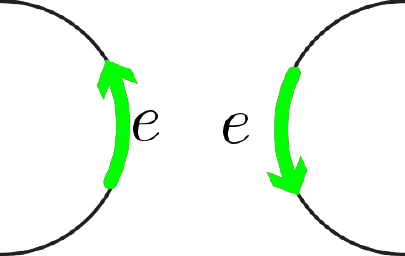}}\;\; \right) \; = \; \;\raisebox{-4mm}{\includegraphics[height=10mm]{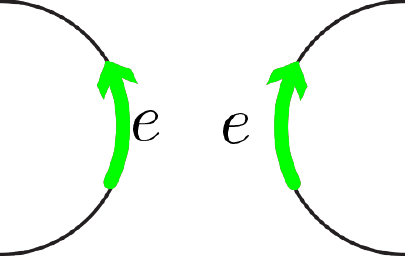}}
\hspace{1cm} \raisebox{0mm}{\text{and}}\hspace{1cm}
\delta\left( \;\; \raisebox{-4mm}{\includegraphics[height=10mm]{a1e}} \;\; \right) \; = \; \; \raisebox{-4mm}{\includegraphics[height=10mm]{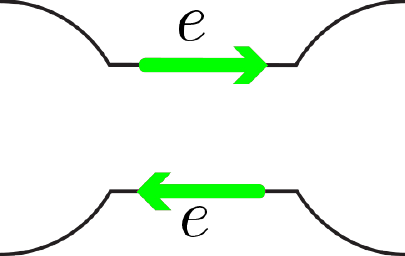}} \]
\caption{The twist operation $\tau$ and partial dual operation $\delta$ given as arrow presentations. }
\label{partials}
\end{figure}

\begin{figure}[ht]
\caption{The full action of $
\delta$  and $\tau$ on a single edge (marked with a * on the left). Note that the four graphs with loop edges each have three vertices, shown in yellow.}
\centering
\includegraphics[width=.8\textwidth]{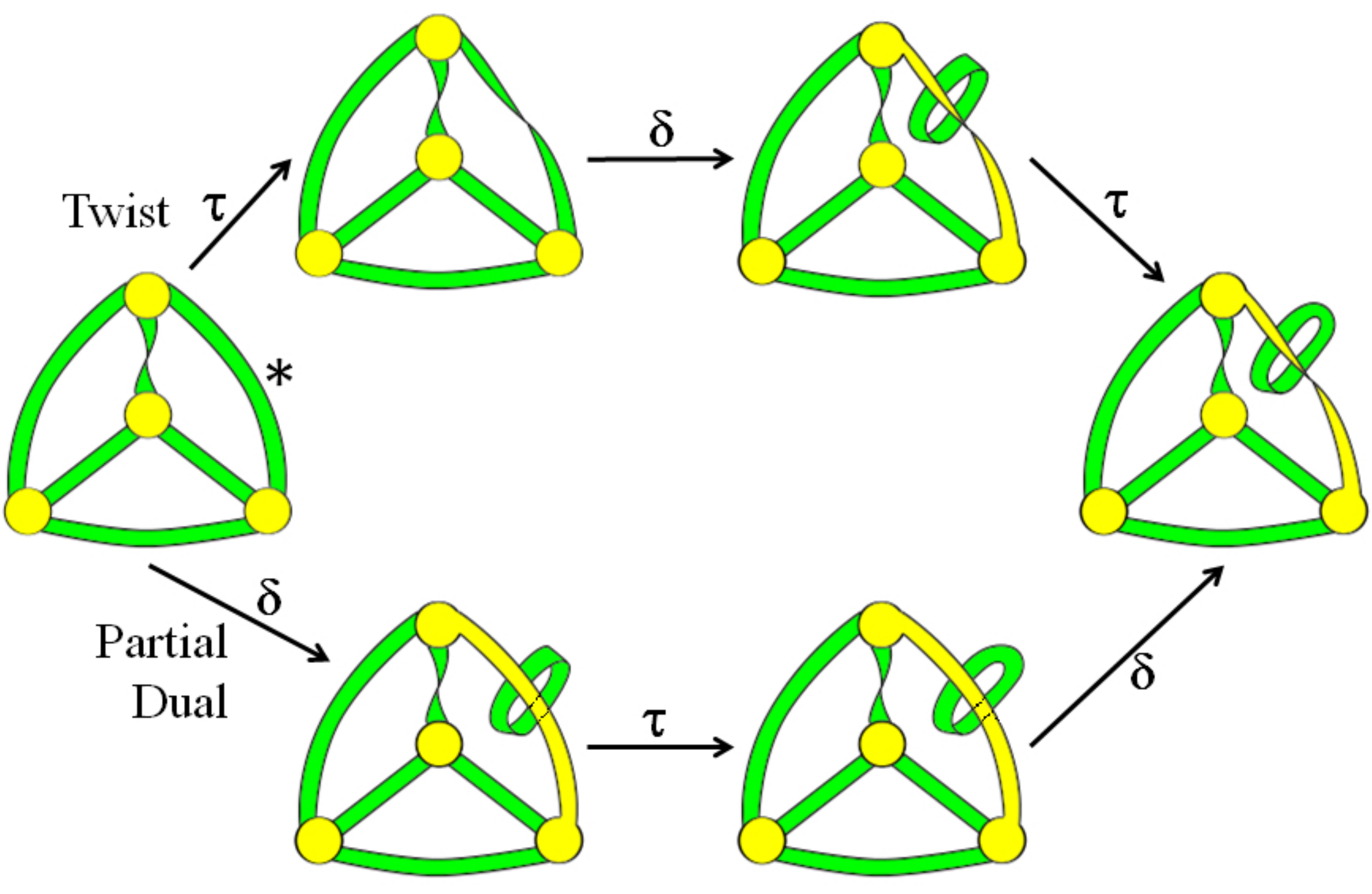}
\label{figure:EdgeOrbit}
\end{figure}

The Wilson group and the Wilson group action of \cite{Wil79} are the forerunners of the ribbon group and restricted ribbon group action from \cite{EMM12, EMM13a} that we extend here.  Wilson noticed that taking the geometric dual and taking the Petrie dual could be thought as operators on embedded graphs (although his attention was mainly on regular maps), and that, although each are of order two, together they generate a group isomorphic to $S_3$, now known as the Wilson group, that acts on embedded graphs.

We use the lexicon from  \cite{EMM12, EMM13a} given in Table \ref{c2.tab1}.

\begin{table}
\begin{tabular}{|l|l|l|l|}
\hline
Generator(s) & Order & Applied to all edges & Applied to a subset of edges  \\
\hline \hline 
$\delta$ & 2&geometric dual   & partial dual \\
\hline 
$\tau$&  2 & Petrie dual  or Petrial   & partial Petrial or twist 
\\
\hline
$\tau\delta\tau$& 2 &Wilson dual or Wilsonial & partial Wilsonial \\ & &(or  \emph{opposite})  & \\ \hline 
$\delta\tau$ or $\tau\delta$ & 3 &triality & partial triality	
\\
\hline 
$\delta$ and $\tau$  & 6   &a direct derivative & twisted dual \\

\hline
\end{tabular}
\bigskip ~
\caption{Operators on embedded graphs or subsets of their edges.} \label{c2.tab1} \end{table}

\section{A new algebraic framework for the ribbon group action}\label{sec.Ribbon Group}

The ribbon group and ribbon group action were defined in \cite{EMM12, EMM13a}.  These constructs fully generalize surface duality and the Wilson group, and thus provide new tools for understanding ribbon graphs. For example, \cite{EMM12} shows that if $G$ and $H$ are graphs, then their medial graphs,  $G_m$ and $H_m$,  are isomorphic as abstract graphs if and only if $G$ and $H$ are twisted duals, while \cite {Chm09, EMM12, EMM13a, EMM13b, EMMpre} give numerous results for topological graph polynomials arising from twisted duality and the ribbon group action, and several authors have explored genus ranges of partial duals in various settings \cite{EZ17, Mof13, Mof16}.  However, the ribbon group action of \cite{EMM12, EMM13a} is limited in that self twuality properties  under it are restricted to only the case of {\it canonical} self-twuality, in which the canonical identification of the edges of the graph and its twual yields an isomorphism.  Here we develop a new algebraic framework for the ribbon group and ribbon group action that facilitates using algebraic tools to expose the more commonally occurring natural twuality properties and to facilitate incorporating the role of graph isomorphism in graph twualities.

Since in the broader setting here, six twisted duality operations apply to individual edges, the edge ordering formalism below keeps track of which operation applies to which edge.  More importantly however, the edge ordering essentially tracks graph isomorphism following a duality operation.  For example, in Figure \ref{figure:digondual}, after a partial dual operation is applied to each edge, so that the net result is the classical dual of the whole graph, the graph and its dual are isomorphic under the map that takes $e$ to $e$ and $f$ to $f$.  However, in Figure \ref{figure:popdual}, the isomorphism between the graph and its dual maps $e$ to $f$ and $f$ to $e$.  We will distinguish between these two kinds of self-duality as canonical and natural, respectively, formalizing this in Definitions \ref{gammadef} and \ref{uniformdef} below.

\begin{figure}[ht]
\caption{A graph that is canonically self-Petrial and also canonically self-dual} 
\centering
\includegraphics[scale=.4]{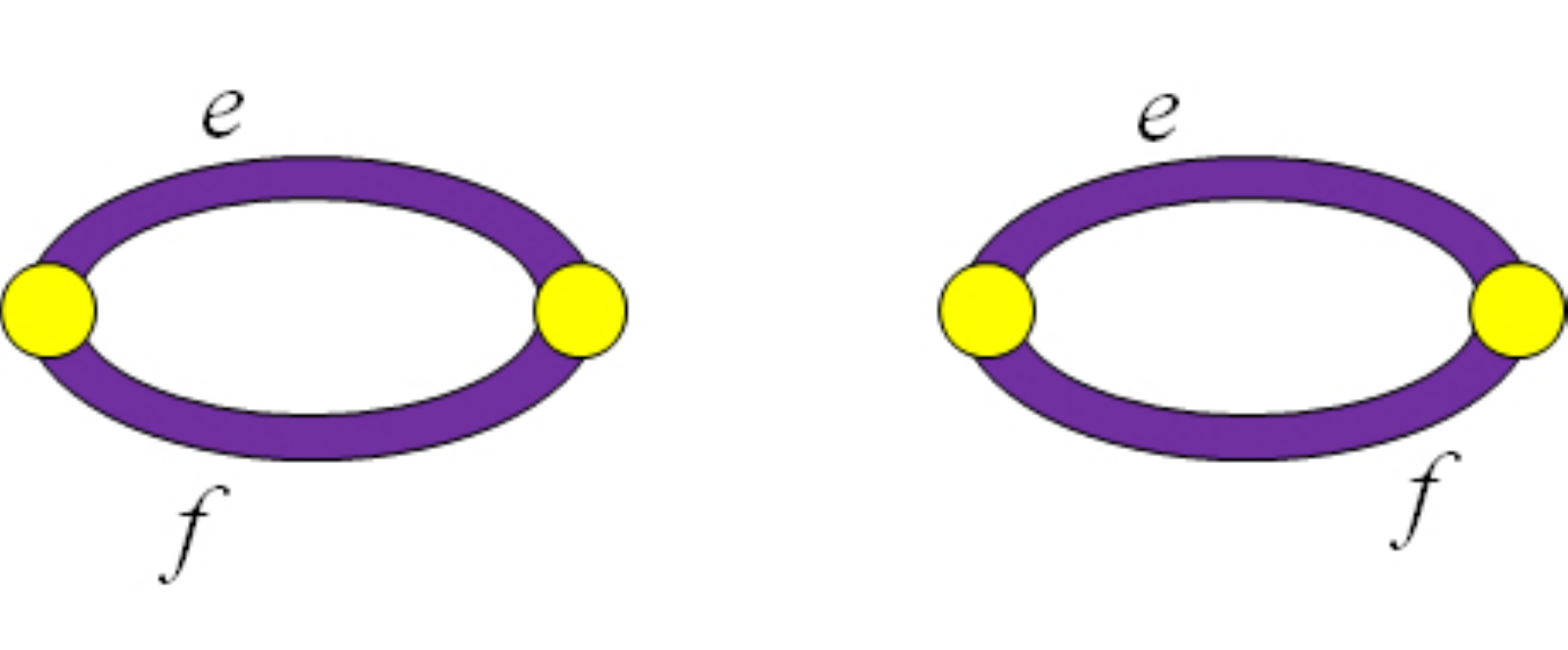}
\label{figure:digondual}
\end{figure}

\begin{figure}[ht]
\caption{A self-dual graph}
\centering
\includegraphics[scale=.4]{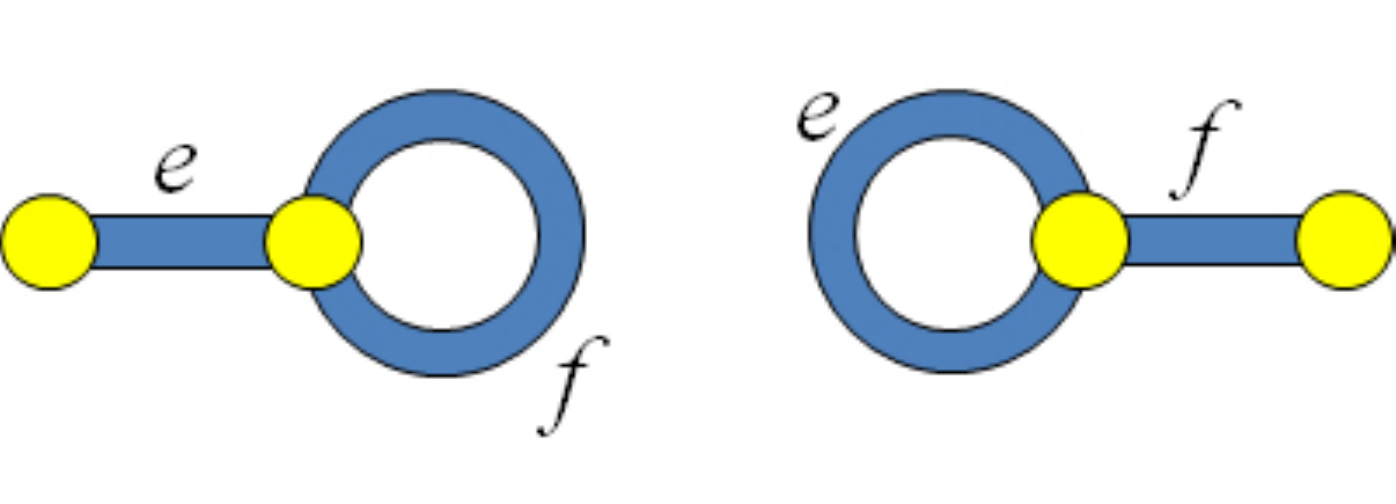}
\label{figure:popdual}
\end{figure}

\begin{figure}[ht]
\caption{A self-Petrial graph}
\centering
\includegraphics[scale=.4]{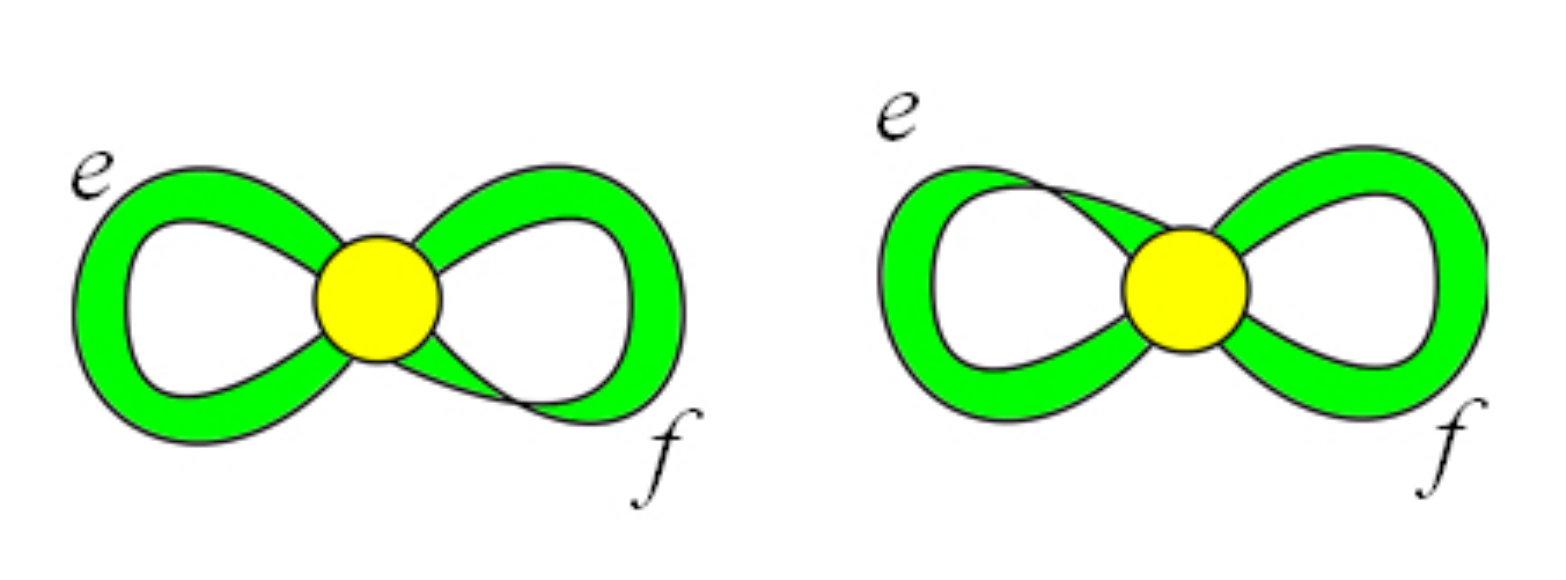}
\label{figure:bowtie}
\end{figure}

The graph in Figure \ref{figure:digondual} is also canonically self-Petrial, since  twists on non-loop edges may propagate through a ribbon graph, resulting in cancellation.  Likewise if one edge of the digon is twisted, the resulting graph is canonically self-Petrial under the map that takes $e$ to $e$ and $f$ to $f$, since ribbon graphs are equivalence classes under vertex flips, so twists on non-loop edges may propagate through a ribbon graph. The graph in Figure \ref{figure:popdual} is self-dual, but not canonically self-dual, since the isomorphism between the graph and its dual is not the canonical identification of edges. The graph in Figure \ref{figure:popdual} is not self-Petrial.     Also, the join of two loops, one twisted and one not, as in Figure \ref{figure:bowtie}, is self-Petrial, but not canonically self-Petrial, as a twist can not move from one loop to another.  In general, no graph with loops can be canonically self-Petrial or canonically self-dual.

Because of this, we will begin by working with graphs with a linear ordering of the edges.
Let $\calG_{(n)}$ denote the set of equivalence classes under isomorphism of ribbon graphs with exactly $n$ edges, and let
\[\calG_{or (n)} := \left\{  (G,\ell) \mid G \in \calG_{(n)} \text{ and } \ell \text{ is a linear ordering of the edges} \right\}\] be the set of equivalence classes of ribbon graphs with  exactly $n$ ordered edges.

In particular, we think of $\ell$ explicitly as a bijection $\ell\colon [n] \mapsto E(G)$, whose input is a position in the ordering, and whose output is the edge in that position.  Two elements $(G, \ell_G)$ and $(H, \ell_H)$ of $\calG_{or (n)}$ are isomorphic if there is an isomorphism of $G$ and $H$ as ribbon graphs so that the bijection $g\colon E_G \rightarrow E_H$ agrees with the linear orderings in that 
$\ell_G(i)= e \iff \ell_H(i) = g(e)$.

The two operations, the twist $\tau$ and partial-dual $\delta$, act on a specified edge $e_i$ of an embedded graph  as shown in Figures \ref{partials} and \ref{figure:EdgeOrbit}.

In \cite{EMM12} it was shown that, for each edge $e$, in $(G,\ell)$, applying $\tau^2$,  $\delta^2$, or $(\tau \delta)^3$  acts as the identity. Thus,  the group \[ \fS := \langle  \delta, \tau   \; |\;   \delta^2, \tau^2, (\tau \delta)^3   \rangle ,\]
which is isomorphic to the symmetric group of order three, acts on any fixed edge of a graph $(G,\ell) \in \calG_{or(n)}$.

This action readily extends to a group action of $\fS^n$ on  $\calG_{or (n)}$ by allowing the various group elements to act on any subset of the edges independently and simultaneously, not just on a single distinguished edge.

If $\bga = ( g_1, g_2,g_3, \ldots , g_n ) \in \fS^n$,  then, for any  $(G,\ell) \in \calG_{or(n)}$, the action of $\bga$ on $(G,\ell)$ applies $g_i$ to the edge $e_i$ in the ordering given by $\ell$.    We will view the indexing as a map, i.e. think of $\bga$ as  $\bga\colon [n] \rightarrow \fS$, so that the action applies the element $\bga (i)=g_i$ to the edge $ \ell (i)$ of $G$.

Frequently twisted duality, in particular partial duals, is given for graphs without ordered edges by specifying which element of $\fS$ is applied to which subset of a 6-partition of the edges.  In particular, Proposition~3.7 of \cite{EMM12} shows that every twisted dual admits a unique expression of the form 
\begin{equation}\label{e.ga}
G^{\prod^6_{i=1}{\g_i(A_i)}},
\end{equation}
where the $A_i$ partition $E(G)$, and where $\g_1=1, \g_2=\tau , \g_3=\delta , \g_4=\tau \delta , \g_5= \delta \tau$, and  $\g_6 = \tau \delta \tau \in \fS$.   This may also be written in the expanded form
\begin{equation}\label{e.gaexp}
G^{\prod_{e \in E(G)}{\g_{s(e)}(e)}},
\end{equation}
where  if $e \in A_i$, then $s(e) =i$.  Often, if only one element  of $\fS$ is applied to some subset $A \subseteq E(G)$, then the notation is simplified to $G^{\g(A)}$ for $\g \in \fS-\{1\}$, e.g.  writing $G^{\delta (A)}$ for the partial dual with respect to $A$.

With the preceding, we can write the action of $\bga$ on $(G,\ell)$ as
\begin{equation}\label{e.closedform}
\bga(G,\ell) = \left(G^{\Gamma(\bga,\ell)}, \ell\right),
\end{equation}
where $\Gamma(\bga,\ell) = \prod^6_{i=1}{\g_i(\ell \bga^{-1}(\g_i))}=\prod_{e \in E(G)}{\left[ \bga \ell^{-1}(e) \right] (e)}$.  Here $\Gamma(\bga,\ell)$ just sorts the edges of $E(G)$ into a 6-partition  according to the operation applied to them by $\bga$,  and then applies the appropriate operation to each partition.  Although the edge order $\ell$ is used to determine $\Gamma(\bga,\ell)$, note that $\ell \bga^{-1}(\g_i)$ is a set, so the result may be applied to $G$, which is unordered, without reference to edge order. 

We make the following observations that will be used to prove associativity in Proposition \ref{semidirect} below.   We first note that permuting the order of the edges translates simply to permuting which element of $\fS$ applies to which edge.  

\begin{equation}\label{shift pi}
\Gamma(\bga,\ell \pi^{-1}) =\prod_{e \in E(G)}{\left[ \bga (\ell \pi^{-1})^{-1}(e) \right] (e)} = \prod_{e \in E(G)}{\left[ (\bga\pi) \ell^{-1}(e) \right] (e)}=\Gamma(\bga \pi,\ell).
\end{equation}
Furthermore, note that composition with a permutation distributes over the group operation in $\fS^n$, so that
\begin{equation}\label{distribute pi}
\bga_1 \pi \circ \bga_2 \pi = (\bga_1 \circ \bga_2) \pi, 
\end{equation}

\begin{equation} \label{iterate Gamma left}
\bga_1 \circ (\bga_2 \pi)  = (\bga_1 \pi^{-1}  \circ \bga_2 ) \pi, 
\end{equation}

and
\begin{equation} \label{iterate Gamma right}
(\bga_1 \pi) \circ \bga_2 = (\bga_1 \circ \bga_2 \pi^{-1}) \pi; 
\end{equation}

also, 
\begin{equation} \label{inverse Gamma pi}
(\bga \pi) ^{-1} = \bga ^{-1}  \pi, \text{ since } \left(\bga\pi \circ \bga^{-1}\pi \right) = \mathbf{1} \pi = \mathbf{1}.
\end{equation}
Lastly, iterated applications of $\Gamma$ are captured by multiplication, in that

\begin{equation}
(G^{\Gamma(\bga_2,\ell)})^{\Gamma(\bga_1,\ell)} = G^{\Gamma(\bga_1\bga_2,\ell)}.
\end{equation}

Caution: Note that with this notation, $(G^{\times})^* = G^{(* \times)}$.

\medskip
There is also a second group action on  $\calG_{or(n)}$.  The symmetric group on $n$ elements, $S_n$, also acts on $\calG_{or(n)}$  by  permuting the edge order, so $\pi (G,\ell) = (G,\ell \pi^{-1})$, where $\ell {\pi}^{-1}$ is just composition of the functions $\ell$ and $\pi$. 
Thus, applying  $\pi$ to  $(G,\ell)$ permutes the ordered list of edges.  We write $ \iota $ for the identity in $S_n$.

We turn to a semidirect product of $\fS^n$ and $S_n$ for an action on $\calG_{or(n)}$ that combines these two actions in a compatible way.  This will be the primary algebraic tool for manipulating ribbon graphs, as its function is to apply specific elements of the ribbon group to specific edges, and furthermore to keep track of isomorphisms between graphs via the edge orderings.

\begin{proposition} \label{semidirect}
Let $\phi\colon S_n \rightarrow  \Aut \fS^n$ by $\phi(\pi) \mapsto \phi_{\pi}$, where $\phi_{\pi}(\bga) = \bga \pi^{-1}$.  Then $\phi$ is a homomorphism, and the semidirect product $\fS^n \rtimes_{\phi}  S_n$ acts on $\calG_{or(n)}$ by $(\bga, \pi)(G,\ell) = \bga(G, \ell \pi^{-1}) =(G^{\Gamma(\bga,\ell\pi^{-1})}, \ell \pi^{-1}) = (G^{\Gamma(\bga \pi,\ell)}, \ell \pi^{-1})$.

\end{proposition}

\begin{proof}

Showing that $\phi$ is a homomorphism is routine since $\phi_{\pi_1\pi_2}(\bga) = \bga (\pi_1 \pi_2)^{-1} = \phi_{\pi_1} \phi _{\pi_2} (\bga)$. Thus we have a well-defined  semidirect product $\fS^n \rtimes_{\phi}  S_n$ with multiplication given by $(\bga_1, \pi_1) (\bga_2, \pi_2) = (\bga_1 \phi_{\pi_1}(\bga_2), \pi_1 \pi_2)$.

It remains to verify associativity, that $(\bga_1,\pi_1)\cdot\left[(\bga_2,\pi_2)\cdot(G,\ell)\right] = \left[ (\bga_1,\pi_1)(\bga_2,\pi_2)\right] \cdot (G,\ell)$.

We give two proofs that the action is associative, as they illustrate how the algebraic and topological machinery may be used.  However, at the heart of both proofs is the observation that first applying $(\bga_2,\pi_2)$ and then applying $(\bga_1,\pi_1)$, requires `undoing' the edge ordering given by $\pi_1$ so that the operations in $\bga_2$ apply to the correct edges.  This  exactly corresponds to the $\pi_1^{-1}$ that appears when first multiplying $(\bga_1,\pi_1)(\bga_2,\pi_2)$ and only then applying the result to $(G,\ell)$, rather than applying the multiplicands iteratively to $(G,\ell)$.

\emph{Proof 1.}
The first proof leverages the identities in Equations (\ref{shift pi}) through (\ref{iterate Gamma right}).  We compute as follows:

\[
\begin{array}{rcl}
(\bga_1,\pi_1)\cdot\left[(\bga_2,\pi_2)\cdot(G,\ell)\right] 

&=&(\bga_1,\pi_1)\cdot\left(G^{\Gamma(\bga_2 \pi_2,\ell)}, \ell \pi_2^{-1}\right) \\[0.25ex]

&=&\left((G^{\Gamma(\bga_2 \pi_2,\ell)})^{\Gamma(\bga_1 \pi_1 ,\ell\pi_2^{-1})}, \ell \pi_2^{-1}\pi_1^{-1}\right) \\[0.25ex]

&=&\left((G^{\Gamma(\bga_2 \pi_2,\ell)})^{\Gamma(\bga_1 \pi_1 \pi_2 ,\ell)}, \ell \pi_2^{-1}\pi_1^{-1}\right) \\[0.25ex]

&=& \left(G^{\Gamma(\bga_1\pi_1\pi_2\circ\bga_2 \pi_2,\ell)}, \ell \pi_2^{-1}\pi_1^{-1}\right) \\[0.25ex]

&=& \left(G^{\Gamma((\bga_1\circ\bga_2 \pi_1^{-1})\pi_1\pi_2,\ell)}, \ell \pi_2^{-1}\pi_1^{-1}\right) \\[0.25ex]

& = & \left(\bga_1 \circ \bga_2\pi_1^{-1}, \pi_1\pi_2\right)\cdot(G,\ell)\\[0.25ex]

& = & \left[ (\bga_1,\pi_1)(\bga_2,\pi_2)\right] \cdot (G,\ell).

\end{array}
\]
Here the third equality follows from Equation \ref{shift pi}, the fourth from Equation \ref{iterate Gamma right}, and the fifth from Equation \ref{distribute pi}.

\emph{Proof 2.}
The second proof uses the following observations, which are reductions of operations within the semigroup action:
\begin{equation} \label {decomp pi}
\begin{array}{rcl}
 (\bga, \iota) \cdot\left[ (\mathbf{1},\pi)\cdot(G,\ell)\right]
 
 & = & \left( [ G^{\Gamma(\mathbf{1},\ell\pi^{-1})} ]^{\Gamma(\bga,\ell\pi^{-1})},\ell\pi^{-1} \right)\\[0.25ex]
 
 & = & \left(G^{\Gamma(\bga,\ell\pi^{-1})}, \ell\pi^{-1} \right) \\[0.25ex]
 
 & = & (\bga,\pi)\cdot(G,\ell),
  \end{array}
 \end{equation}

\smallskip
 
 \begin{equation}
 \begin{array}{rcl}

(\mathbf{1},\pi)\cdot\left[ (\bga\pi, \iota) \cdot (G,\ell)\right]

& = & \left( [ G^{\Gamma(\bga\pi,\ell)} ]^{\Gamma(\mathbf{1},\ell\pi^{-1})},\ell\pi^{-1} \right)\\[0.25ex]

& = & (\bga,\pi)\cdot(G,\ell), 
\end{array}
\end{equation}

\smallskip

\begin{equation}
(\mathbf{1},\pi_1\pi_2)\cdot(G,\ell) = \left(G^{\Gamma(\mathbf{1},\ell\pi_2^{-1}\pi_1^{-1})}, \ell\pi_2^{-1}\pi_1^{-1} \right) = (\mathbf{1},\pi_1)\cdot\left[(\mathbf{1},\pi_2)\cdot(G,\ell)\right], \text{and}
\end{equation}

\smallskip

\begin{equation}
(\bga_1\bga_2,\iota)\cdot(G,\ell) = (\bga_1,\iota)\cdot\left[(\bga_2,\iota)\cdot(G,\ell)\right].
\end{equation}
In fact, each of these is a particular instance when the associativity of the action holds, although we don't need that observation per se. 

Now we have
\[
\begin{array}{rcl}
(\bga_1,\pi_1)\cdot\left[(\bga_2,\pi_2)\cdot(G,\ell)\right]

& = & (\bga_1,\pi_1)\cdot\left[  (\mathbf{1},\pi_2)\cdot\left[ (\bga_2\pi_2, \iota) \cdot (G,\ell)\right] \right]\\

& = & (\mathbf{1},\pi_1)\cdot\left[ (\bga_1\pi_1, \iota) \cdot\left[  (\mathbf{1},\pi_2)\cdot\left[ (\bga_2\pi_2, \iota) \cdot (G,\ell)\right]\right] \right]\\

& =& (\mathbf{1},\pi_1)\cdot\left[ (\mathbf{1},\pi_2) \cdot\left[ (\bga_1\pi_1\pi_2, \iota) \cdot\left[ (\bga_2\pi_2, \iota) \cdot (G,\ell)\right]\right] \right]\\

& = & (\mathbf{1},\pi_1\pi_2) \cdot \left[ (\bga_1\pi_1\pi_2 \circ \bga_2\pi_2, \iota) \cdot (G,\ell)\right]\\

& = & (\bga_1 \circ \bga_2\pi_1^{-1}, \pi_1\pi_2)\cdot(G,\ell)\\

& = & \left[ (\bga_1,\pi_1)(\bga_2,\pi_2)\right] \cdot (G,\ell)
\end{array}
\]
\end{proof}

\medskip

Recall that applying a permutation to an indexed set applies the inverse permutation to the indices, so that $\phi_{\pi}$ acts by permuting the indices of $\bga$, where if $\bga = ( g_1, g_2,g_3, \ldots , g_n )$, then $\phi_{\pi}(\bga)= ( g_{\pi^{-1}(1)}, g_{\pi^{-1}(2)},g_{\pi^{-1}(3)}, \ldots , g_{\pi^{-1}(n)} )$.

Also note that $(\bga, \pi)^{-1}=(\bga^{-1} \pi, \pi^{-1})$, applying Equation \ref{distribute pi} as needed.

\medskip
We close this section by establishing some notation that will facilitate our exploration of the orbits and stabilizers of the ribbon group action.

As usual $Orb((G,\ell_G)) = \{(H, \ell_H) \mid (\bga, \pi)(G, \ell_G) =(H, \ell_H) \text{ for some } (\bga, \pi) \in \fS^n \rtimes_{\phi}  S_n \}$, although we will usually write $Orb(G,\ell_G)$ for $Orb((G,\ell_G))$.  In addition, we will often focus particularly on the action of the subgroup $\fS^n \rtimes_{\phi}  \iota$, and will denote its orbit as $Orb_{\iota}(G,\ell_G) = \{(H, \ell_H) \mid (\bga, \iota)(G, \ell_G) =(H, \ell_H) \text{ for some } \bga \in \fS^n \}$.

We will see that the various twualities an embedded graph may exhibit may be revealed by examining stabilizers of different kinds.

\begin{definition}\label{gammadef} We say that an embedded graph $(G,\ell)$ is \emph{self-$\bga$} if $\bga$ is not the identity and there exists $\pi \in S_n$ such that  $(\bga, \pi)(G, \ell) = (G,\ell)$, and we say that $(G,\ell)$ is \emph{canonically self-$\bga$} if $(\bga, \iota)(G, \ell) = (G,\ell)$.
\end{definition}

In particular, $(G,\ell)$ is self-$\bga$ for some $\bga$ different from the identity if it has a non-trivial stabilizer in $\fS^n \rtimes  S_n$, and canonically self-$\bga$ if it has a non-trivial stabilizer in  $\fS^n \rtimes  \iota$.

We will be most interested in the case that $G$ is self-dual, self-Petrial, self-trial, etc. in the traditional sense. This corresponds to $\bga$ being \emph{uniform}, that is, every entry of $\bga$ is the same.

\begin{definition}\label{uniformdef} We say that an embedded graph $(G,\ell)$ is \emph{self-twual} via $\pi$ (or just self-twual) if there is a uniform $\bga \in \fS^n$ and a $\pi \in S_n$ such that  $(\bga, \pi)(G, \ell) = (G,\ell)$.   We say that $(G,\ell)$ is \emph{canonically self-twual} if $(G,\ell)$ is self-twual via $\iota$. 
\end{definition}
Thus, for example, a graph $G$ is canonically self-dual if  $(\boldsymbol{\delta}, \iota)(G, \ell) = (G,\ell)$.

Note that because $\calG_{or(n)}$ is a set of equivalence classes,   $(\bga, \pi)(G, \ell) = (G,\ell)$ means that there is an isomorphism between $G^{\Gamma (\bga,\ell\pi^{-1})}$ and $G$, and  
the correspondence of edges under this isomorphism is given by the mapping $\ell\pi^{-1}(i) \mapsto \ell(i)$.

\remark{A comparison with the ribbon group and ribbon group action on graphs with ordered edges of  \cite[Definition 2.21]{EMM13a} shows that this action is exactly the action of $\fS^n \rtimes  \iota$.  The permutation group in $\fS^n \rtimes  S_n$ allows us to define natural self-twuality precisely as $(\bga, \pi)(G, \ell) = (G,\ell)$, with $\pi$ specifying the isomorphism.  However, the limitation of the cannonical action of \cite{EMM13a}  is evident in \cite[Theorem 2.25]{EMM13a} where the most that can be said is that two graphs have a given self-twuality if they both belong to a particular set, with the isomorphism unspecified.}

\section{Propagation of twuality} \label{propagation}

We show here that if a graph $H$ is self-$\bga$ for any non-trivial $\bga$, then any graph $G$ in its orbit is also self-$\bga'$ for some $\bga'$.  Thus, once any graph is found to have some twisted duality property, that property propagates through its whole orbit.  Furthermore, as we will see later, it possible to search the orbit efficiently for graphs with desired self-twuality properties. 

Recall that in the classical setting, we get new self-twualities by conjugation, so that if $G$ has some self-twuality property and $H$ is an element of the orbit of $G$ under the Wilson group action, then $H$ will have a self-twuality that is conjugate to $G$'s.  For example, if $G$ is self-dual, with $G=G^{\delta(E)}$ and $H=G^{\tau(E)}$, then $H$ is self-$\tau \delta \tau^{-1}$.

This same principle holds for any pair of elements in the Wilson group.   Here we adapt this principle to the more refined setting of the ribbon group action, which then allows us to  study twuality systematically in the subsequent sections.

We begin with a short technical lemma to verify the intuitively clear fact that if $H$ is self-$\bga$, then this property is preserved if the edge order and the entries of $\bga$ are permuted simultaneously by the same permutation.   

\begin{lemma} \label{selfgamma}  If $(H, \ell)$ is self-$\bga$ via $\mu$, then for any permutation $\pi$,  $(H, \ell \pi^{-1})$  is self-$\bga \pi^{ -1}$ via $\pi \mu \pi^{-1}$.

\end{lemma}

\begin{proof}
It suffices to verify the commutativity of the following diagram:
\[
\xymatrixcolsep{6pc}
\xymatrix{ 
(H,\ell \pi^{-1}) \ar[r]^-{\left(\bga \pi^{-1},\pi \mu \pi^{-1} \right)}  & (H,\ell \pi^{-1})\\
(H,\ell_H) \ar[u]^{(\mathbf{1},\pi)}   
\ar[r]^-{(\bga, \mu)} 
& (H,\ell_H) \ar[u]_{(\mathbf{1},\pi) }
}
\]
By associativity of the action, this follows from the following calculation in $\fS^n \rtimes_{\phi}  S_n$.
  \[
  \begin{array}{rcl}
  (\bga \pi^{-1}, \pi \mu \pi^{-1} ) (\mathbf{1}, \pi)
  
  &=&  (\bga \pi^{-1} \cdot \mathbf{1}\pi\mu^{-1}\pi^{-1}, \pi \mu  )\\
  
  &=&  (\bga\pi^{-1} \cdot \mathbf{1}, \pi\mu) \\
 
  &=& (\mathbf{1}\cdot\bga\pi^{-1},\pi\mu)\\
 
   & = & (\mathbf{1},\pi)(\bga,\mu).
   \end{array}
  \]
\end{proof}

The following proposition gives the main result in the simplest setting, where there are no permutations of the edges to keep track of, and hence the principle of the proof is easier to see.

\begin{proposition}\label{canonical prop}
Suppose $(H, \ell_H)$ is  canonically self-$\bga$, and $(G, \ell_G) \in Orb_{\iota}(H, \ell_H)$ with $(\bal, \iota)(H,\ell_H)=(G, \ell_G)$.  Then  $G$ is also canonically self-$\bga'$ where $\bga'$ is the result of conjugation, specifically $\bga'=\bal \bga \bal^{-1}$.
\end{proposition}

\begin{proof}
  Associativity of the action and a simple calculation in $\fS^n \rtimes_{\phi}  S_n$ verify that the following diagram is commutative, which proves the result.
\[
\xymatrixcolsep{6pc}
\xymatrix{ 
(G,\ell_G) \ar[r]^-{\left(\bal\bga \bal^{-1},\iota\right)}  & (G,\ell_G)\\
(H,\ell_H) \ar[u]^{(\bal,\iota)}   
\ar[r]^-{(\bga, \iota)} 
& (H,\ell_H) \ar[u]_{(\bal,\iota) }
}
\]
\end{proof}

A consequence of Proposition \ref{canonical prop}, given in the following theorem, is that canonical self-twuality is propagated throughout the entire orbit of a graph.

\begin{theorem} \label{canonical equivs}

Let $(F, \ell_F)$ be a graph.  Then the following are equivalent:

\begin{enumerate}
\item \label{st member} There is a graph $(H, \ell_H) \in Orb_{\iota}(F, \ell_F)$ that is canonically self-twual, i.e. self-dual, -Petrial, or -Wilsonial (respectively, self-trial);
\item \label{st every} Every graph $(G, \ell_G)$ in $Orb_{\iota}(F, \ell_F)$ is canonically self-$\bga'$ where every element of $\bga'$ has order 2 (respectively, 3); 
\item \label{st one} There exists a graph $(G, \ell_G)$ in $Orb_{\iota}(F, \ell_F)$ that is canonically self-$\bga'$ where every element of $\bga'$ has order 2 (respectively, 3). 
\end{enumerate}
\end{theorem}

\begin{proof}
The proof uses the conjugacies  for $S_3$ given in Table \ref{table:cong}.

\begin{table}
\renewcommand\arraystretch{1.3}
\setlength\doublerulesep{0pt}
\begin{tabular}{r||*{6}{c|}}
$\al\ga\al^{-1}$ & 1 &$ \tau $&$ \delta $&$ \tau\delta $&$ \delta\tau $&$ \tau\delta\tau$  \\
\hline\hline
1 & 1&$ \tau $&$ \delta $&$ \tau\delta $&$ \delta\tau $&$ \tau\delta\tau$  \\ 
\hline
$\tau$ &$ 1$ &$ \tau $&$ \tau \delta \tau$ & $\delta \tau$ & $\tau \delta$ & $\delta$  \\ 
\hline
$\delta$ & 1 &$\tau \delta \tau$ &$ \delta $&$ \delta \tau $&$ \tau \delta $&$ \tau $ \\ 
\hline
$\tau\delta$ & 1&$ \delta $&$ \tau \delta \tau $&$ \tau \delta $&$ \delta \tau $&$  \tau$ \\ 
\hline
$\delta \tau$ & 1 &$\tau \delta \tau $&$ \tau $&$ \tau \delta $&$ \delta \tau $&$ \delta $\\ 
\hline
$\tau \delta \tau$ & 1 &$\delta $&$ \tau $&$ \delta \tau $&$ \tau \delta $&$ \tau \delta \tau$\\ 
\hline
\end{tabular}
\vskip 5mm
\caption{Conjugacy Table ($\al$ labels the rows, $\ga$ the columns)}
\label{table:cong}
\end{table}

To see that Item \ref{st member} implies Item \ref{st every}, suppose that $(H, \ell_H)$ is a canonically self-twual graph via uniform $\bga$.  Then by Theorem \ref{canonical prop}, every graph $(G, \ell_G)$ in $Orb_{\iota}(H, \ell_H)=Orb_{\iota}(F, \ell_F)$ is canonically self-$\bal \bga \bal^{-1}$ for some $\bal$.  However, conjugation preserves order here, so by Table \ref{table:cong}, if $\bga$ is $\delta-$, $\tau-$, or $\tau \delta \tau-$ uniform, then each element in $\bal \bga \bal^{-1}$ is in $\{\tau, \delta, \tau \delta \tau \}$.  Similarly, if $(G, \ell_G)$ is canonically self-trial, then every element of  $\bal \bga \bal^{-1}$ must be in $\{ \tau \delta, \delta \tau \}$.  Note that $\bal$ may be trivial,  in which case $\bga =\bga'$ and still has every element of order 2 or 3.

\medskip

That Item \ref{st every} implies Item \ref{st one} is immediate.  

\medskip

To show that Item \ref{st one} implies Item \ref{st member}, assume $(G, \ell_G)$ is a graph in $Orb_{\iota}(F, \ell_F)$ that is canonically self-$\bga'$ where every entry of $\bga'$ is in $\{\tau, \delta, \tau \delta \tau \}$.  We may then use the conjugacy table to select elements of $\bal$ so that $\bal \bga' \bal^{-1}$  is $\delta-$, $\tau-$, or $\tau \delta \tau-$ uniform. With this, $(H, \ell_H) = (\bal, \iota)(G, \ell_G)$ is the canonically self-dual, -Petrial, or -Wilsonial graph we seek, since if $\bga$ is the desired twuality, then $(\bga, \iota) (H, \ell_H) = (\bga, \iota) (\bal, \iota)(G, \ell_G) = (\bal, \iota)(\bga', \iota)(\bal, \iota)^{-1} (\bal, \iota)(G, \ell_G)=(\bal, \iota)(\bga', \iota)(G, \ell_G)=(\bal, \iota)(G, \ell_G)=(H, \ell_H)$.  An analogous approach proves the self-trial case. 
\end{proof}

The general case where $(G, \ell_G)$ is any graph in the orbit of $(H, \ell_H)$ (that is, in $Orb(H, \ell_H)$, not necessarily $Orb_{\iota} (H, \ell_H)$) and where $(H, \ell_H)$ is self-$\bga$  but not necessarily canonically so, is essentially the same, but involves keeping track of the permutations. Although the conjugation is somewhat obfuscated by the appearance of the permutations, we are simply reordering to be sure that the conjugation is applied to the correct edges at each step.

\begin{theorem}\label{natural prop}
Suppose $(H, \ell_H)$ is self-$\bga$  via $\mu$ and that $(G, \ell_G) \in Orb(H, \ell_H)$ with $(\bal, \pi)(H,\ell_H)=(G, \ell_G)$.  Then $(G, \ell_G)$ is self-$\bga'$ via $\mu'$ where  $\bga'$ is a `near-conjugate' of $\bga$ by $\bal$, with 
$$\bga' = \bal \cdot \bga \pi^{-1} \cdot \bal^{-1} \pi \mu^{-1} \pi^{-1},$$  and $ \mu'= \pi \mu \pi^{-1}.$  
\end{theorem}

\begin{proof}
It suffices to verify the commutativity of the following diagram:
\[
\xymatrixcolsep{10pc}
\xymatrix{ 
(G,\ell_G) \ar[r]^-{\left(\bal \cdot \bga \pi^{-1} \cdot \bal^{-1} \pi \mu^{-1} \pi^{-1}, \pi \mu \pi^{-1}\right)}  & (G,\ell_G)\\
(H,\ell_H) \ar[u]^{(\bal,\pi)}   
\ar[r]^-{(\bga, \mu)} 
& (H,\ell_H) \ar[u]_{(\bal,\pi) }
}
\]

By associativity of the action, this follows from the following calculation in $\fS^n \rtimes_{\phi}  S_n$.
  \[
  \begin{array}{rcl}
  
  \left(\bal \cdot \bga \pi^{-1} \cdot \bal^{-1} \pi \mu^{-1} \pi^{-1}, \pi \mu \pi^{-1}\right)(\bal,\pi)
  &=&   \left(\bal \cdot \bga \pi^{-1} \cdot \bal^{-1} \pi \mu^{-1} \pi^{-1} \cdot \bal\pi\mu^{-1}\pi^{-1}, \pi \mu \pi^{-1}\right)\\
  
  &=&  (\bal \cdot \bga\pi^{-1},\pi\mu) \\
 
  &=& (\bal,\pi)(\bga,\mu).
  
   \end{array}
  \]
\end{proof}

Since $\fS$ is not commutative, we write $\prod_{i=m}^{1} {\g_{i}} $ in the following theorem statement to indicate the product $\g_m \g_{m-1} \ldots \g_1$.  We also write $ord(g)$ for the order of a group element in $\fS$.

\begin{theorem} \label{natural equivs}  
Let $(F, \ell_F)$ be a graph, $g \in \fS$, and $\mu$ a permutation.  Then the following are equivalent:
\begin{enumerate}
\item \label{wk member} There is a graph $(H, \ell_H) \in Orb(F, \ell_F)$  that is  self-$\bga = (g, g, \ldots , g)$  via the permutation $\mu$, i.e. $(H, \ell_H)$ is self-dual, -Petrial, -Wilsonial, or -trial;     

\item \label{wk every} Every graph $(G, \ell_G) \in Orb(F, \ell_F)$ is self-$\bga'= (\g_1, \g_2, \ldots, \g_n)$ via $\mu'$, where $\mu'=\pi \mu \pi^{-1}$ for some $\pi$, and where $\bga'$  has the property that if $C=(c_1, c_2, \ldots, c_m)$ is a cycle in the cycle decomposition of $\mu'$, then $ord (\prod_{i=m}^{1} {\g_{c_i}} ) = ord(g ^m)$  ;

\item \label{wk one} 
There exists a graph $(G, \ell_G) \in Orb (F, \ell_F)$ that is self-$\bga'= (\g_1, \g_2, \ldots, \g_n)$ via $\mu'$, where $\mu'=\pi \mu \pi^{-1}$ for some $\pi$, and where $\bga'$  has the property that if $C=(c_1, c_2, \ldots, c_m)$ is a cycle in the cycle decomposition of $\mu'$, then 
$ord (\prod_{i=m}^{1} {\g_{c_i}} ) = ord(g ^m)$.

\end{enumerate}
\end{theorem}

\begin{proof}
To see that Item \ref{wk member} implies Item \ref{wk every}, suppose that $(H, \ell_H)$ is a  self-twual graph via uniform $\bga$ and permutation $\mu$, and $(G, \ell_G)$ is a graph in $Orb(H, \ell_H)=Orb (F, \ell_F)$.  Then, by Theorem \ref{natural prop}, $(G, \ell_G)$ is  self-$\bga' = \bal \cdot \bga \pi^{-1} \cdot \bal^{-1} \pi \mu^{-1} \pi^{-1} $  via $\mu'= \pi \mu \pi^{-1}$ where  $\bal$ and $\pi$ satisfy $(\bal, \pi)(H, \ell_H) = (G, \ell_G)$. In particular, for this $\bal$ and $\pi$,  we have that
\begin{equation}\label{alphagamma}
(\bga, \mu) \, = \, 
(\bal, \pi)^{-1}(\bga', \mu')(\bal, \pi) \, = \,
\left((\bal^{-1} \cdot \bga' \cdot \bal\mu'^{-1})\pi,\, \pi^{-1}\mu'\pi\right).
\end{equation}

Hence, $(\bal^{-1} \cdot \bga' \cdot \bal\mu'^{-1})\pi=\bga$, and so $(\bal^{-1} \cdot \bga' \cdot \bal\mu'^{-1})=\bga\pi^{-1}$.  However, $\bga$ is uniform, so $ \bga\pi^{-1}= \pi$, and we have that
\begin{equation} \label{setcycles}
    (\bal^{-1} \cdot \bga' \cdot \bal\mu'^{-1})=\bga.
\end{equation}

We now write $\bga'=(\g_1, \g_2, \ldots, \g_n)$ and also write  $\bal=(\al_1, \al_2, \ldots, \al_n)$, recalling that  $\bal\mu'^{-1}= (\al_{\mu'^{-1}(1)}, \al_{\mu'^{-1}(2)}, \ldots, \al_{\mu'^{-1}(n)})$.  Considering each entry of Equation \ref{setcycles} individually,  this yields that 
\begin{equation}\label{entries} 
\al_i^{-1} \cdot \g_i \cdot \al_{\mu'^{-1}(i)}=g  \text{ for all } i.
\end{equation}
  In particular, if $C=(c_1, c_2, \ldots, c_m)$ is a cycle in the cycle decomposition of $\mu'$ (we use the convention that a cycle begins with its lowest number), then Equation \ref{entries} becomes
  \begin{equation} \label{cycleconj}
    \al_{c_i}^{-1} \cdot \g_{c_i} \cdot \al_{ c_{i-1}}    = g  \text{ for all } i \in \{1, 2, \ldots,m\}, \text { where we write } c_0= c_m.  
  \end{equation}
  
  Thus, $\al_{ c_{m-1} }=\g^{-1}_{c_m} \cdot \al_{c_m} \cdot g$ , and $\al_{c_{m-2}}= \g^{-1}_{c_{m-1}} \cdot \g^{-1}_{c_m} \cdot \al_{c_m }\cdot g^2$,    and in general 
  \begin{equation} \label{iterate}
  \al_{c_{m-i}}= \left ( \prod_{j=m-i+1}^{m}{\g_{c_j}^{-1}} \right ) \cdot \al_{c_{m}} \cdot g^i.
\end{equation}  
  
  When $i=m$, it follows that $\al_{c_{m}}= \left ( \prod_{j=1}^{m}{\g_{c_j}^{-1}} \right ) \cdot \al_{c_{m}} \cdot g^m$   and hence $\left ( \prod_{j=m}^{1}{\g_{c_j} }\right ) =\al_{c_{m}} \cdot g^m \cdot \al_{c_{m}}^{-1}$.   Since conjugation preserves order in $\fS$, this proves the implication.

\medskip
That Item \ref{st every} implies Item \ref{wk one} is immediate.  
\medskip

To show that Item \ref{wk one} implies Item \ref{wk member},
suppose $(G,\ell_G) \in Orb (F, \ell_F)$ is self-$\bga'= (\g_1, \g_2, \ldots, \g_n)$ via $\mu'=\pi \mu \pi^{-1}$ for some $\pi$, and that $\bga'$  has the property that if $C=(c_1, c_2, \ldots, c_m)$ is a cycle in the cycle decomposition of $\mu'$, then 
$ord (\prod_{i=m}^{1} {\g_{c_i}} ) = ord(g ^m)$.  

We observe that we can assume $\mu'=\mu$, i.e. that we can take $\pi = \iota$.  This follows from Lemma \ref{selfgamma}, which says that we can simply re-order the edges of $(G,\ell_G)$ so that $G$ is self-$\bga' \pi$ via $\mu=\mu'$.    We then note that since in $\prod_{i=m}^{1} {\g_{c_i}}$ the indices are from a cycle of $\mu'$, it follows that $ord (\prod_{i=m}^{1} {\g_{\mu'(c_i})}  = ord(\prod_{i=m}^{1} {\g_{c_{i+1}}} ) = ord (\prod_{i=m}^{1} {\g_{c_i}} ) = ord(g ^m)$.

Thus, Item \ref{wk one} implies that there is some $(G,\ell_G) \in Orb (F, \ell_F)$ that is self-$\bga'= (\g_1, \g_2, \ldots, \g_n)$ via $\mu'=\mu$, and that $\bga'$  has the property that if $C=(c_1, c_2, \ldots, c_m)$ is a cycle in the cycle decomposition of $\mu'$, then 
$ord (\prod_{i=m}^{1} {\g_{c_i}} ) = ord(g ^m)$.

We will construct $\bal \in \fS^n$ as follows.  Given a cycle $C$ of length $m$ in the cycle decomposition of $\mu'$, since $ord (\prod_{j=m}^{1} {\g_{c_j}} ) = ord(g ^m)$, we may use the conjugacies in Table \ref{table:cong} to choose $\al_{c_m}$ such that $ \al_{c_{m}} \cdot \left( \prod_{j=m}^{1}{\g_{c_j} }\right ) \cdot \al_{c_{m}}^{-1} = g^m$.  Then, for $1\leq i < m$, recursively define  $\al_{c_i}=g^{-1} \cdot \al_{c_{i+1}} \cdot \g_{i+1}$. Note that we have $\al_{c_i}=g^{-(m-1)} \cdot \al_{c_m} \cdot \left( \prod_{j=m}^{i+1}{\g_{c_j} }\right )$ for $1\leq i < m$, and also that $\al_{c_1}=g \cdot \al_{c_{m}} \cdot \g^{-1}_{1}$.

Then, for all $i$, and reading indices of the $c_i$'s mod $m$, we find that 
\begin{equation}\label{alpha g}
  g= \al_{c_{i+1}} \cdot \g_{i+1} \cdot \al^{-1}_{\mu'^{-1}(c_{i+1})}.  
\end{equation}

We repeat this, choosing some suitable $\al_{c_m}$ for each cycle of length $m$ in $\mu'$ for all lengths $m$.

Since Equation \ref{alpha g} holds for all entries and $\mu = \mu'$, it follows that $(\bga, \mu)=(\bal, \iota)(\bga', \mu') (\bal, \iota)^{-1}.$

With this, $(H, \ell_H) = (\bal, \iota)(G, \ell_G)$ is the desired  self-dual, -Petrial, -Wilsonial, or respectively self-trial graph, since then  $(\bga, \mu)(H, \ell_H) = (\bga, \mu)(\bal, \iota)(G, \ell_G) = (\bal, \iota)(\bga', \mu') (\bal, \iota)^{-1}(\bal, \iota)(G, \ell_G) = (\bal, \iota)(\bga', \mu')(G, \ell_G)=(\bal, \iota)(G, \ell_G)=(H, \ell_H)$.
\end{proof}

Observe that $\bal$, as constructed in the proof of Theorem \ref{natural equivs} is not uniquely determined, as each cycle of $\mu'$ gives rise to several possibilities. Each of the resulting $\bal$'s can give a different self-twual graph $H$, although some of them might possibly be isomorphic to one another.

\medskip
The impact of Theorems \ref{canonical equivs} and \ref{natural equivs} is that if we find any graph $G$ that is self-$\bga$ via a permutation $\mu$ for any $\bga$ and $\mu$, then we can quickly test for the existence of an $\bal$ so that $\bga $ has the form $\bal \bga' (\bal^{-1} \mu^{-1})$ for some uniform $\bga'$, and thus identify self-twual graphs in the orbit of $G$.  We will see this in action in the following sections, starting in Section~\ref{sec:OEB} by identifying graphs for which it is reasonable to test for being self-$\bga$ for some $\bga$, and then in Section \ref{Code discussion} giving a polynomial time algorithm to determine the existence of an $\bal$.

\medskip
Note that in the special case that $\mu=\iota$, Theorems \ref{natural prop} and \ref{natural equivs} reduce to Theorems \ref{canonical prop} and \ref{canonical equivs}, respectively. In the case of Theorem \ref{canonical prop},  we have that $\bga'$ is trivial if and only if $\bga$ is.  However, in Theorem \ref{natural prop}, if $\bga = \mathbf{1}$ but $\mu \neq \iota$ then $H$ has a non-trivial automorphism group, with $\mu \in \Aut\!(H)$.  In this case, it is possible that $\bga'$ may be non-trivial, which leads to the following theorem and a new strategy for generating self-twual graphs.

\begin{theorem} \label{automorph duals} Suppose the permutation $\mu$ corresponds to an element of the automorphism group of $H$ (as an embedded graph), that is,  $(\mathbf{1}, \mu)(H, \ell_H)=(H, \ell_H)$. Then for every ribbon group element $\bal$ we have that $(G, \ell_G) = (\bal, \iota)(H, \ell_H)$  is self-$\bga'$ via $\mu$ with $\bga'= \bal \cdot \bal^{-1} \mu^{-1}$, and $\bga'$ non-trivial exactly when  $\bal \neq \bal \mu^{-1} $.

\end{theorem}

\begin{proof}

This follows from Theorem \ref{natural prop} with $\bga = \mathbf{1}$ and $\pi = \iota$, and the observation that $(\bal^{-1} \mu^{-1})^{-1} = \bal \mu^{-1} $. 
\end{proof}

Thus, if a graph $H$ has a non-trivial automorphism group, as is the case for example with regular maps, then there is potential for finding self-twual graphs in its orbit by seeking solutions to  $\bga'= \bal \cdot \bal^{-1} \mu^{-1}$ for uniform  $\bga'$.  We give a small example  applying Theorem \ref{automorph duals} in Example \ref{ex:automorph example}, and use Theorem \ref{automorph duals} to give an infinite family of graphs that are self-trial but neither self-dual nor self-petrial in Subsection \ref{automorph examples}.

\begin{example} \label{ex:automorph example}
Let $H$ be the bouquet with six edges shown on the left of Figure \ref{figure:automorph example}.  Let $\mu$ be the automorphism that rotates $H$ by $120^{\circ}$.  We solve $\bal \cdot \bal^{-1} \mu^{-1}=\{\delta \tau, \delta \tau,\delta \tau,\delta \tau,\delta \tau,\delta \tau \}$ to get $\bal = \{1, 1, \delta \tau, \delta \tau, \tau \delta, \tau \delta\}$.  With this, $(G, \ell_G) = (\bal, \iota)(H, \ell_H)$, shown on the right of Figure \ref{figure:automorph example}, is self-trial.  Since $G$ has a single loop it cannot be self-Petrial, and hence it also cannot be self-dual.
\end{example}
\begin{figure} 
{\includegraphics[height=35mm]{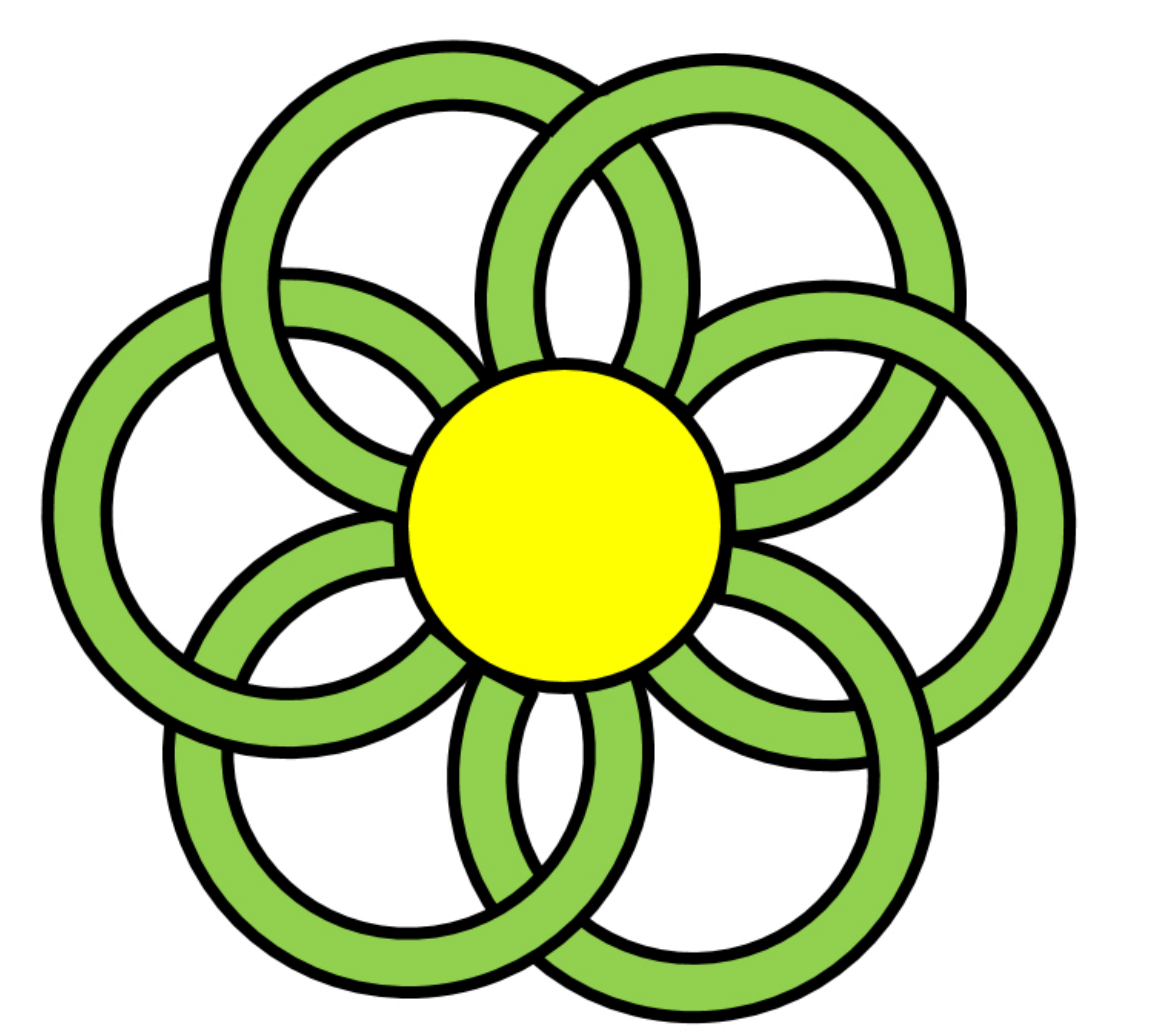}} 
\hspace{1cm} 
{\includegraphics[height=40mm]{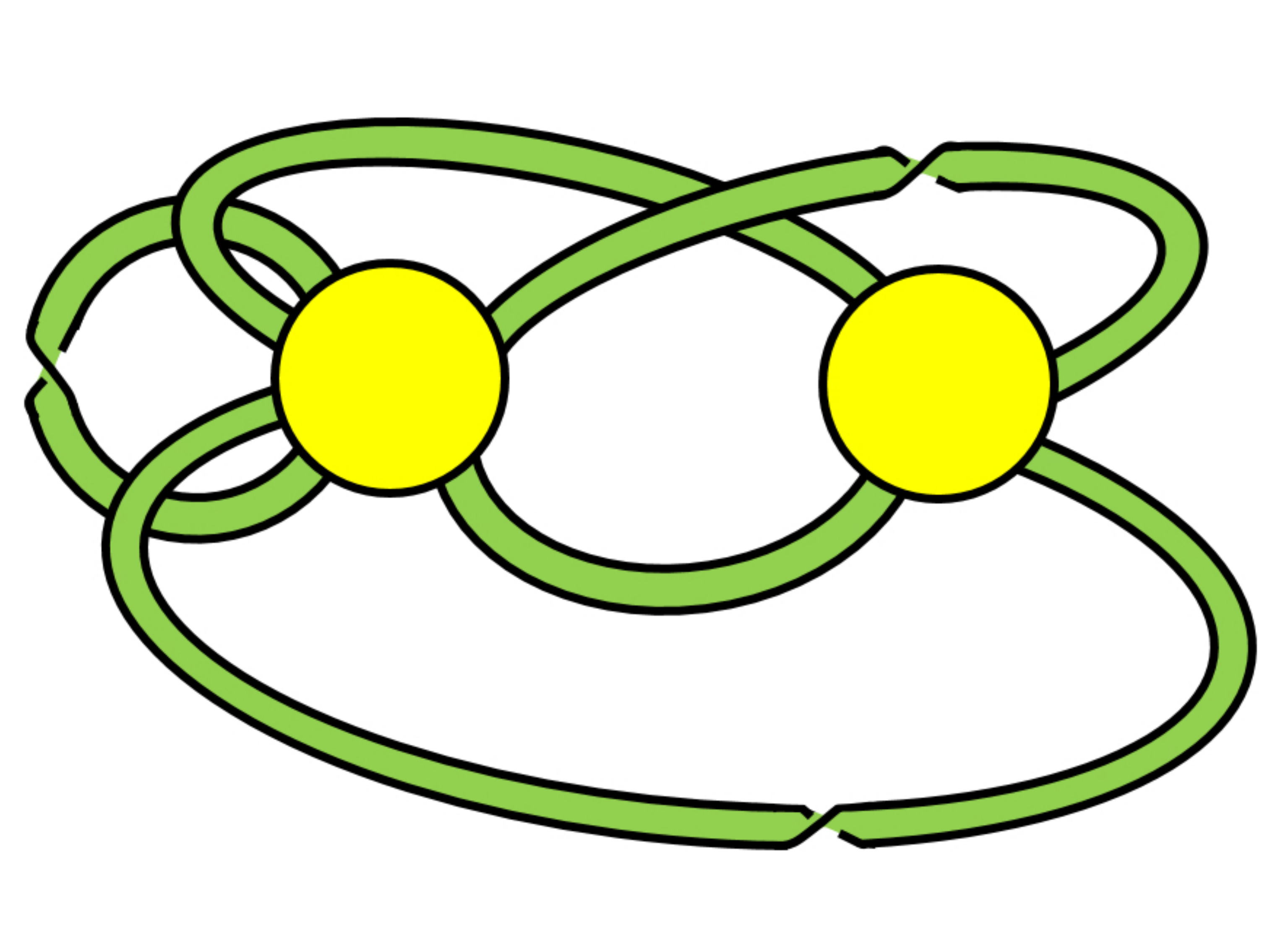}}
\caption{A graph $H$ (left) with non-trivial automorphism group, and a graph $G$ (right) in its orbit that is self-trial but neither self-dual nor self-petrial. }
\label{figure:automorph example}
\end{figure}

\section{Single vertex orientable ribbon graphs}\label{sec:OEB}

In this section we show that the search for self-twual graphs may be reduced to studying the self-$\bga$ properties of single-vertex orientable graphs, \emph{i.e.} orientable embedded boquets (OEBs).  Since, by the results of Section \ref{propagation},  any OEB encodes the twisted duality properties of every graph in its orbit, this provides a new approach for searching for self-dual, self-Petrial, self-trial, etc. graphs. 

We first note that every graph has an OEB in its orbit.  Thus, OEBs provide representatives of the orbits under the ribbon group action that may be more readily analyzed than general graphs.

\begin{proposition}\label{OEBexist}
If $(G,\ell_G)$ is a connected graph, then there is an OEB in its orbit.
\end{proposition}
\begin{proof}
We induct on the number of vertices of $G$. Suppose that $(G,\ell_G)$ has more than one vertex. Since $G$ is connected, it has some non-loop edge, say $\ell_G(i)$. Define $\bga_1\in\fS_n$ to have  $\bga_1(i)=\delta$ and $\bga_1(j) = 1$ for $j\neq i$, and define $(G_1,\ell_G) = (\bga_1,\iota)\cdot(G,\ell_G)$. Note that $G_1$ has one fewer vertex than $G$ does. Iterating this process, we obtain a bouquet $(G_k,\ell_G) = (\bga_k,\iota)\cdot\ldots \cdot(\bga_1,\iota)\cdot(G,\ell_G)$ in the orbit of $(G,\ell_G)$.

If $G$ is orientable then we are done. If not, then because $G$ has only a single vertex we can unambiguously identify a subset $S =\{ \ell(i_1),\ldots, \ell(i_r) \}$  of edges of $(G_k,\ell_G)$ which are twisted. Defining $\bga' \in \fS_n$ to have $\bga'(i) = \tau$ for $\ell(i)\in S$ and $\bga'(i) = \iota$ otherwise, we obtain an OEB $(\bga',\iota)\cdot(G_k,\ell_G)$ in the orbit of $(G,\ell_G)$.   
\end{proof}

We now set the stage for searching for self-twual graphs.  We again begin with the special case of being canonically self-$\bga$ and canonically self-twual, as here results can be found simply by checking a conjugation table.

\begin{proposition} \label{canonical OEB rep} 
Suppose $H$ is an OEB that is canonically self-$\bga$ for some non-trivial $\bga$.  Then every graph in its orbit is also canonically self-$\bga'$ where $\bga'$ is a conjugate of $\bga$.  
\end{proposition}

\begin{proof}
This follows immediately from Theorem \ref{canonical prop} in the special case that $H$ is an OEB.
\end{proof}

The following Theorem \ref{canonical all} gives a practical application of Theorem \ref{canonical equivs} to generating all canonically self-twual graphs by confirming that determining the canonically self-$\bga$ properties of an OEB indeed captures all possible canonically self-twual graphs in its orbit.

\begin{theorem}\label{canonical all} Every canonically self-dual, -Petrial, or -Wilsonial (respectively, self-trial) graph $G$ is in the orbit of an OEB $H$ that is canonically self-$\bga'$ for some $\bga'$ of order 2 (respectively, 3).
\end{theorem}

\begin{proof}
This follows from Theorem \ref{canonical equivs} by taking $G=F$, recalling that $Orb(G)=Orb(H)$ for any $H \in Orb(G)$, and from Proposition \ref{OEBexist}  that tells us that there is always an OEB $H$ in $Orb(G)$.
\end{proof}

The following two results are the analogs of Proposition \ref{canonical OEB rep} and Theorem \ref{canonical all} where the twuality is not necessarily canonical.  The proofs, which use Theorem \ref{natural equivs} instead of Theorem \ref{canonical equivs}, with some attention to technical details since the `conjugation' for the proof of Theorem \ref{natural all} involves a permutation,  are left to the reader.

\begin{proposition}\label{natural OEB rep} 
Suppose $H$ is an OEB that is self-$\bga$ for some non-trivial $\bga$, via a permutation $\mu$.  Then every graph in $Orb(H)$ is also self-$\bga'$ via $\mu'$, where  $\bga'$ is the `near-conjugate' $$\bga' = \bal \cdot \bga \pi^{-1} \cdot \bal^{-1} \pi \mu^{-1} \pi^{-1} $$ of $\bga$,  and $\mu'= \pi \mu \pi^{-1}$.
  
\end{proposition}

\begin{theorem}\label{natural all} Every  self-dual, -Petrial, or -Wilsonial (respectively, self-trial) graph $G$ is in the orbit of an OEB $H$ that is  self-$\bga'$  via $\mu'$ for some $\bga'$ with the property that if $C=(c_1, c_2, \ldots, c_m)$ is a cycle in the cycle decomposition of $\mu'$ , then 
$ord (\prod_{i=m}^{1} {\g_{c_i}} ) = ord(g ^m)$. 
\end{theorem}

Thus, to search for self-twual graphs, we can generate OEBs and test for self-$\bga$ properties.  If we find an OEB $H$ that is self-$\bga$  with the elements $\bga$ have the required order properties, we can then just look in the conjugacy table to generate the  self-twual graphs in its orbit. If we have exhaustively found all $\bga$ such that $H$ is  self-$\bga$, then we will be able generate all of the self-twual graphs in its orbit.  This process and the algorithm for it is discussed further in Sections \ref{Search results} and \ref{Code discussion}.

\section{New graphs from old} \label{Search results}

Here we use the results of the previous sections and a computer search to determine all graphs on up to 7 edges that are self-trial, but neither self-dual nor self petrial. We then apply Theorem \ref{automorph duals} to generate a new infinite family of self-trial graphs that are neither self-dual nor self-petrial.

We represent a graph $(G,\ell_G)$ as a collection of cyclically ordered tuples $[r_1,r_2,\ldots,r_n]$, each tuple conveying the information for attaching ribbon-ends to one of the vertices of $G$. The entries $r_i$ are integers whose absolute value $|r_i|$ is the label of a ribbon and such that $r_i<0$ exactly when ribbon $r_i$ has a twist. Such representations are referred to in Section \ref{svogs_as_chordiags} as being in ``end-label form." Further details about the computer implementation are discussed in Section \ref{Code discussion}.

\subsection {All small self-trial graphs}

Figure \ref{fig:self-svog} lists a few self-$\bga$ OEBs. There are in fact many others for the indicated numbers of ribbons, but these are the ones needed for Figure \ref{fig:selfies}. Figure \ref{fig:selfies} lists all examples, up to isomorphism, of self-trial non-self-dual (and hence also non-self-Petrial) graphs $G$ having $n$ ribbons for $3\leq n\leq 7$. In each case, these examples arose as $(G,\ell_G) = (\bal,\iota)(H,\ell_H)$ for multiple self-$(\bga,\pi)$ OEBs $(H,\ell_H)$, for various  $(\bga,\pi)$, but only one such OEB is listed. The graph $G$ itself is self-trial under the action of $(\boldsymbol{\delta}\boldsymbol{\tau},\sigma)$ for the given element $\sigma$.

\begin{figure}
    \[
    \begin{array}{cll}
         n & \mbox{the OEB} & \mbox{is invariant under}  \\\hline
         3 & H_3 = [1, 2, 3, 1, 2, 3 ] &(\, (1, \delta, \delta, ) , (1\, 2\, 3) )\, )\\[5pt]
5 & H_5 = [1, 2, 3, 4, 2, 5, 4, 1, 5, 3 ] &(\, (\tau\delta, \delta\tau, \tau\delta, 1, \delta\tau, ) , (3 \, 5\, 4 )\, )\\[5pt]
6 & H_6  = [1, 2, 3, 4, 5, 6, 2, 4, 1, 5, 3, 6 ] &(\, (\tau\delta\tau, 1, \tau\delta\tau, \tau\delta, \delta, \tau\delta\tau, ) , (1 \, 6 \,  2 )(3 \,  4\,  5 )\, )\\[5pt]
7 & H_7 = [1, 2, 3, 4, 5, 6, 2, 7, 3, 5, 1, 4, 6, 7 ] &(\, (\tau\delta\tau, \tau\delta, \delta, \delta\tau, \tau\delta, \tau\delta, \delta\tau, ) , (1 \, 6\,  3 )\, )\\
    \end{array}
    \]
    \caption{Some self-$\bga$ OEBs.  }
    \label{fig:self-svog}
\end{figure}

\begin{figure}

\[
\small{
\begin{array}{llll}
n & G &  = \bal H & \sigma\\ \hline
3 & [1, -3, 2, 1, 2, -3 ]  & = (\tau\delta\tau, \tau\delta, \delta )\,H_3 & ( 1, 2, 3 )  \\[5pt]
5 & [1, 4, 2, 3, -5, 2, 1, -5, 4, 3 ]  & = (\tau\delta\tau, \tau\delta, \tau\delta\tau, \tau\delta\tau, \delta )\,H_5 & ( 3, 5, 4 )  \\
5 & [-1, 2, 5, -3, 2, -4, -1, 5, -4, -3 ]  & = (\tau, \tau\delta, \tau\delta\tau, \tau\delta\tau, \delta )\,H_5 & ( 3, 5, 4 ) \\
5 & [-1, 4, 2, 3, -5, 2 ] [-1, -5, 4, 3 ]  & = (\delta, \tau\delta, \tau\delta\tau, \tau\delta\tau, \delta )\,H_5 & ( 3, 5, 4 )  \\
5 & [1, 4, 2, 1, 5, 4, -3 ] [2, 5, -3 ]  & = (\tau\delta\tau, \delta\tau, \tau\delta\tau, \tau\delta\tau, \delta )\,H_5 & ( 3, 5, 4 )  \\[5pt]
6 & [1, -4, -5, 3, -2 ] [1, -5, 6, 3, -4, -2, 6 ]  & = (\tau\delta, \tau, \tau\delta\tau, \tau\delta\tau, \tau\delta, \tau\delta\tau )\,H_6 & \hspace{-2em}( 1, 6, 2 )( 3, 4, 5 )  \\
6 & [-1, 4, -5, -3, -2, 4, -3, 6, -2, -1, -5, 6 ]  & = (\delta\tau, \tau\delta\tau, \tau\delta\tau, \tau\delta\tau, \tau\delta, \delta )\,H_6  & \hspace{-2em}( 1, 6, 2 )( 3, 4, 5 )  \\[5pt]
7 & [1, -5, -6, 4, -3, 7, -6, -2, 1, 4, -5, -3, -2, 7 ]  & = (\tau\delta\tau, \tau\delta\tau, \tau\delta, \tau\delta, \tau\delta\tau, \tau\delta\tau, \tau\delta )\,H_7 &  ( 1, 6, 3 )  \\
7 & [-1, -5, 6, -2, -1, -4, -5, 3, -4, 6, -7, 3, -2, -7 ]  & = (\delta\tau, \tau\delta\tau, \tau\delta\tau, \tau\delta, \tau\delta\tau, 1, \tau\delta )\,H_7 & ( 1, 6, 3 )  \\
7 & [-1, -5, 6, -7, -3, -2, -7, -1, -4, -5, -3, -4, 6, -2 ]  & = (\delta, \tau\delta\tau, \delta\tau, \tau\delta, \tau\delta\tau, \delta, \tau\delta )\,H_7 & ( 1, 6, 3 )  \\
7 & [1, -5, -6, 4, -3, 7, -6, 2, 7 ] [1, 4, -5, -3, 2 ]  & = (\tau\delta\tau, \tau, \tau\delta, \tau\delta, \tau\delta\tau, \tau\delta\tau, \tau\delta )\,H_7 & ( 1, 6, 3 )  \\
7 & [1, -5, 6, 4, -3, 2 ] [1, 4, -5, -3, 7, 6, 2, 7 ]  & = (\tau\delta, \tau, \tau, \tau\delta, \tau\delta\tau, \delta\tau, \tau\delta )\,H_7 & ( 1, 6, 3 )  \\
7 & [1, -2, 3, 4, -6, -7, 3, 5, 4, 1, 5, -6, -2, -7 ]  & = (\tau, \tau, 1, \tau\delta, \tau\delta\tau, \tau, \tau\delta )\,H_7 & ( 1, 6, 3 )  \\
7 & [-1, 5, -6, -4, 3, 7, -6, 2, 3, 5, -4, -1, 2, 7 ]  & = (\tau\delta\tau, \delta, \tau\delta, \tau\delta, \tau\delta\tau, \tau\delta\tau, \tau\delta )\,H_7 & ( 1, 6, 3)  \\
7 & [1, 5, -6, 2, 3, 7, -6, -4, 3, 5, -4, 1, 2, 7 ]  & = (\delta\tau, \delta, \tau\delta\tau, \tau\delta, \tau\delta\tau, 1, \tau\delta )\,H_7 & ( 1, 6, 3 )  \\
7 & [-1, -2, 7 ] [-1, 4, 5, 3, 7, -6, 4, 3, -2, -6, 5 ]  & = (\tau, \delta, 1, \tau\delta, \tau\delta\tau, \tau, \tau\delta )\,H_7 & ( 1, 6, 3 )  \\
7 & [-1, -5, 6, 4, -5, 3, -2, 6, -7, 3, 4, -1, -2, -7 ]  & = (\tau\delta\tau, \delta, \tau\delta, \delta\tau, \tau\delta\tau, \tau\delta\tau, \tau\delta )\,H_7 & ( 1, 6, 3 )  \\
7 & [1, 5, -6, 2, 3, 7, -6, 4, 1, 2, 7 ] [3, 5, 4 ]  & = (\delta\tau, \delta, \tau\delta\tau, 1, \tau\delta\tau, 1, \tau\delta )\,H_7 & ( 1, 6, 3 )  \\
7 & [-1, 5, 4, -1, -2, -7 ] [-2, -6, -7, 3, 4, -6, 5, 3 ]  & = (\tau\delta\tau, \delta, \tau\delta, \tau\delta, \delta, \tau\delta\tau, \tau\delta )\,H_7 & ( 1, 6, 3 )  \\
\end{array}
}
\]
\caption{Up to isomorphism, all self-trial non-self-dual graphs with $3\leq n \leq 7$. In all cases we have $G = (\boldsymbol{\delta}\boldsymbol{\tau},\sigma)G$.}\label{fig:selfies}
\end{figure}

\subsection{Examples from Theorem \ref{automorph duals}} \label{automorph examples}

We now use Theorem \ref{automorph duals} to generate an infinite family of graphs that are self-trial but neither self-dual nor self-petrial.  This family differs from the infinite family of such graphs given in \cite{JP10} which contains very large regular maps in that this family starts with quite small graphs and the members are not generally regular maps. 

Motivating the example below is the observation that twuality acts independently on the components of a one-point join of two graphs.  Thus, if $G$ is an OEB, then the one-point join of the three graphs $G$, $G^{(* \times)}$, and $G^{(\times *)}$ will automatically be self-trial.  The example below is much more general, as it is not a one-point join of graphs.

We begin with the OEB $(H_k,\ell_{H_k})$ on $3k$ untwisted edges $e_1, \ldots, e_{3k}$, where $e_i=\ell_{H_k}(i)$, which are attached according to the following [cyclic] pattern: 
\[
e_2, e_1, e_3, e_2, e_4, e_3, e_5, e_4, \ldots, e_i, e_{i-1}, e_{i+1}, e_i, \ldots, e_{3k}, e_{3k-1}, e_1, e_{3k}.
\]
Defining $\bal$ to be the ribbon group element
\[
\bal = ( \, \overbrace{{{\mathbf 1}, {\mathbf 1}, \ldots, {\mathbf 1}}}^{k \mbox{\ {\scriptsize times}}},  \overbrace{\delta\tau, \delta\tau, \ldots, \delta\tau}^{k \mbox{\ {\scriptsize times}}}, \overbrace{\tau\delta,\tau\delta, \ldots, \tau\delta}^{k \mbox{\ {\scriptsize times}}} \, ),
\]
we obtain the graph $(\bal, \iota)(H_k, \ell_{H_k})$ depicted in Figure \ref{fig:regular_to_selftrial}. Note the labeling of the edge-ends; we will make heavy use of this labeling below.
\begin{figure}[ht]
\caption{The graph $(\bal, \iota)(H_k, \ell_{H_k})$ as described in the text, depicted as an arrow diagram. For each $i$ the ends of edge $e_i$ are labeled $(i)_a$ and $(i)_b$, respectively.}
\centering
\medskip
\includegraphics[width=\textwidth]{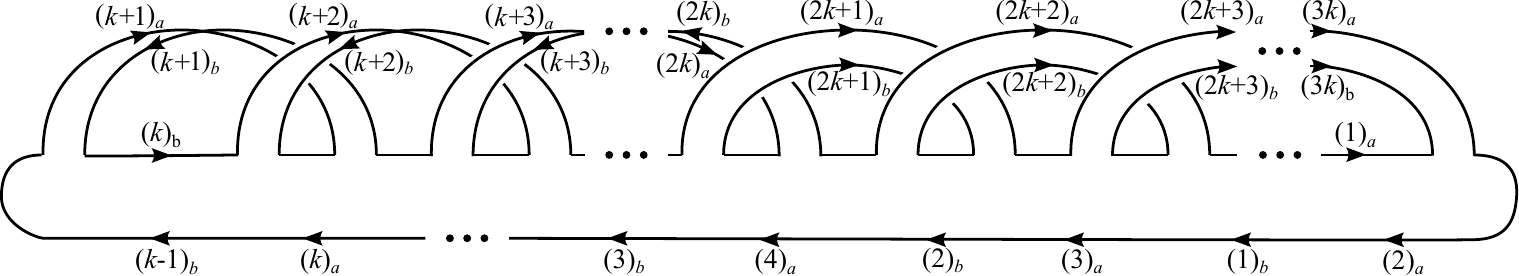}
\label{fig:regular_to_selftrial}
\end{figure}

\begin{proposition}\label{prop:infinite examples}
The graph  $(\bal, \iota)(H_k, \ell_{H_k})$ depicted in Figure \ref{fig:regular_to_selftrial} is self-trial but not self-dual.
\end{proposition}
\begin{proof}
Fix $k$ and write $(G, \ell_G)$ for $(\bal, \iota)(H_k, \ell_{H_k})$ Clearly, the permutation $\mu$ given by $\mu(i) = i+k \mod 3k$ gives an automorphism of $(H_k, \ell_{H_k})$. With $\bal$ defined as above we have 
\[
\bal \cdot \bal^{-1} \mu^{-1} = ( \,  \overbrace{\delta\tau, \delta\tau, \ldots, \delta\tau}^{3k \mbox{\ \scriptsize times}} \, ).
\]
Corollary \ref{automorph duals} thus tells us that the graph $(G, \ell_G)$ is self-trial. To show that  $(G,\ell_G)$ is not self-dual it suffices to show that it is not self-Petrial, since $\delta\tau$ and $\tau$ generate $\fS$.

The graph $(G,\ell_G)$ has one or more vertices, depicted in Figure \ref{fig:regular_to_selftrial} as one or more cyclic strings of labeled arrows. We cut these at the tails of the arrows $(k)_b$, $(k+1)_a$,  $(2k)_a$, $(3k-1)_a$,$(3k)_a$ and $(3k)_b$ and at the heads of the arrows $(2k-1)_b$ and $(2k)_b$, to obtain the following linear ``strands," where we use $(i)\inv_x$ to indicate that the arrow is traversed backwards in the sequence of the strand. (Note that the tail of $(k+1)_a$ is the head of $(k+1)_b$.)
\[
\begin{array}{rcl}
S_{k+1,a} & := &  (k+1)_a \ (k+2)_b\inv \ (k+4)_a \ (k+5)_b\inv \ (k+7)_a \cdots \\ \\

S_{k+1,b} & := & (k+1)_b\inv \ (k+3)_a \ (k+4)_b\inv \ (k+6)_a \ (k+7)_b\inv \cdots \\ \\ 

S_{k,b} & := & (k)_b \ (k+2)_a \ (k+3)_b\inv \ (k+5)_a \ (k+6)_b\inv \ (k+8)_a \cdots \\ \\

S_{2k-1,b} & := & (2k-1)_b\inv \ (2k+1)_a \ (2k+3)_a \ (2k+5)_a \ (2k+7)_a \cdots \\ \\

S_{2k,a} & := & (2k)_a \ (2k+1)_b \ (2k+2)_b \ (2k+3)_b \ (2k+4)_b \cdots \\ \\

S_{2k,b} & := & (2k)_b\inv \ (2k+2)_a \ (2k+4)_a \ (2k+6)_a \ (2k+8)_a \cdots \\ \\

S_{3k,a} & := & (3k)_a \ (2)_a \ (1)_b \ (3)_a \ (2)_b \ (4)_a \cdots (k-1)_b \\ \\ 

S_{3k,b} & := & (3k)_b \ (1)_a\inv \ (3k-1)_a\inv
\end{array}
\]
By construction, these strands account for all edge ends, but the particular sequence of strands found in $(G,\ell_G)$ depends on the value of $k \mod 6$. In each case we provide the simplest argument we could find.

\mycase{$k \equiv 1 \mod 6$.} There is one vertex, given by the cyclic concatenation $$S_{3k,a} \ S_{k+1,a} \  S_{2k,a} \ S_{3k,b} \ S_{2k,b}\inv \ S_{k+1,b}\inv \ S_{k,b} \ S_{2k-1.b}.  
$$
Since there are $3k$ loops and $k$ is odd, it is not possible for $(G,\ell_G)$ to have the same number of twisted loops as untwisted loops. Thus $(G,\ell_G)$ is not self Petrial in this case.

\mycase{$k \equiv 5 \mod 6$.} There is one vertex, given by the cyclic concatenation
$$
S_{3k,a} \ S_{k+1,a} \ S_{2k,b} \ S_{3k,b}\inv \ S_{2k,a}\inv \ S_{k,b}\inv \ S_{k+1,b} \ S_{2k-1,b}.
$$
As for the case $k \equiv 1$, $(G,\ell_G)$ is not self-Petrial in this case.

\mycase{$k \equiv 2 \mod 6$.} There are two vertices, given by the cyclic concatenations $$S_{3k,a} \ S_{k+1,a} \ 
S_{2k,b}$$
and
$$S_{k+1,b} \ S_{2k-1,b} \ S_{3k,b}\inv \ S_{2k,a}\inv \ S_{k,b}\inv.$$
We have $k = 6m+2$ for some $m>0$. Let $v$ denote the vertex corresponding to the concatenation $S_{3k,a}S_{k+1,a} \ S_{2k,b}$ and let $w$ denote the other vertex. Note that the only loops on $v$, namely edges $2,3,4, \ldots, k-1$, have both their ends in $S_{3k,a}$, and these are all untwisted. This is a total of $k-2 = 6m$ loops on $v$. In all, there are 3 additional edges ($3k, 1$, and $k$) attached in $S_{3k,a}$, $2(\frac{k+1}{3})$ attached in $S_{k+1,a}$, and $\frac{k-1}{2}$ attached in $S_{2k,b}$.  Thus, the number of loops on $w$ is
\[
3k - \left[ (k-2)+3 + 2\left(\frac{k+1}{3}\right) + \frac{k-1}{2} \right] \ = \ 5m+1.
\]
Since the Petrial of $(G,\ell_G)$ has a vertex with $6m$ twisted loops whereas $(G,\ell_G)$ itself does not, we see that $(G,\ell_G)$ is not self-Petrial for any $k \equiv 2 \mod 6$.

\mycase{$k \equiv 3 \mod 6$.} There are two vertices, given by the cyclic concatenations $$S_{3k,a} \ S_{k+1,a} \ S_{2k-1,b}$$
and
$$S_{k+1,b} \ S_{2k,a} \ S_{3k,b} \ S_{2k,b}\inv S_{k,b}\inv.$$

Let $k=6m+3$ for some $m\geq 0$, let $v$ denote the verex corresponding to the concatenation $S_{3k,a} \ S_{k+1,a} \ S_{2k-1,b}$ and let $w$ denote the other vertex. As in the case $k \equiv 2$, the vertex $v$ has $k-2$ loops, all untwisted; in the case $k\equiv 3$ this is $6m+1$ untwisted loops. Also attached to $v$ are additional edges as follows: $(k-2)+3$ attached to $S_{3k,a}$, $2\frac{k}{3}-1$ attached to $S_{k+1,a}$, and $\frac{k+1}{2}$ attached to $S_{2k-1,b}$. The number of loops on $w$ is then 
\[
3k - \left[ (k-1)+3 + 2\frac{k}{3}-1 + \frac{k+1}{2} \right] \ = \ 5m+2.
\]
As in the case $k\equiv 2$, we see that $(G,\ell_G)$ is not self-Petrial for any $k \equiv 3 \mod 6$.

\mycase{$k \equiv 0 \mod 6$.} There is one vertex, given by the cyclic concatenation $$S_{3k,a} \ S_{k+1,a} \  S_{2k-1,b} \ S_{3k,b}\inv \ S_{2k,a}\inv \ S_{k+1,b}\inv \ S_{k,b} \ S_{2k,b}.$$

We will show that, along the single vertex, there is a sequence  of length $2k-2$ of consecutive untwisted edge-ends, whereas the longest sequence of consecutive twisted edge ends is at most of length $\frac{4}{3}k-1$. This discrepancy guarantees that $(G,\ell_G)$ is not self-Petrial. 

To verify the claim, we analyze as follows. First, note that in strand $S_{3k,a}$, in the sequence from $(2)_a$ to $(k-1)_b$, we have both ends of the $k-2$ untwisted loops $2, 3, 4, \ldots, k-1$ as well as $(1)_b$ and $(k)_a$. The other end of loop $1$ is in $S_{3k,b}\inv$, and thus $(1)_b$ is an untwisted edge end, and the other end of loop $k$ is at the beginning of $S_{k,b}$, and thus $(k)_a$ is also an  untwisted edge end. In all, then, we have in $S_{3k,a}$ a total of $2(k-2)+2 = 2k-2$ consecutive untwisted edge ends. 

We now bound the length of any sequence of consecutive twisted edge ends. From the discussion above we know that edges $1, 2,$ and $k-1$ are untwisted. It is also true that edge $2k$ is  untwisted since end $(2k)_a\inv$ in $S_{2k,a}\inv$ is paired with $(2k)_b\inv$ in $S_{2k,b}$. We use the ends of these edges to break up the cyclic sequence of all edge-ends into five subsequences we denote I, II, III, IV, V as follows:
\[
(1)_a \ \mbox{I} \ (2k)_a\inv \ \mbox{II} \ (2k)_b\inv \ \mbox{III} \ (2)_a \ \mbox{IV} \ (k-1)_a \ \mbox{V} \ (1)_a.
\]
Of course, any sequence of consecutive twisted edge ends must lie entirely within one of the numbered subsequences other than IV, which has only untwisted edge ends. We can thus complete the proof by determining the lengths of I, II, III, and V, respectively, and observing that all are less than $2k-2$.

\begin{itemize}
    \item[I:] The end $(3k)_b\inv$ in $S_{3k,b}\inv$ together with the $k-1$ edge-ends of $S_{2k,a}\inv$ other than $(2k)_a\inv$ yield $k$ edge ends in I. \medskip
    
    \item[II:] The $\frac{2}{3}k -1$ edge-ends from $S_{k+1,b}\inv$ together with the $\frac{2}{3}k$ edge-ends from $S_{k,b}$ give $\frac{4}{3}k -1$ edge-ends in II.  \medskip
    
    \item[III:] There are a total of $\frac{1}{2}k$ edge-ends in III, namely, $\frac{1}{2}k-1$ edge-ends other than $(2k)_b\inv$ from $S_{2k,b}$, together with $(3k)_a$ in $S_{3k,a}$. \medskip
    
    \item[V:] In V there are a total of $\frac{2}{3}k + \frac{1}{2}k$ edge-ends. These are the $\frac{2}{3}k-1$ edge-ends in $S_{k+1,a}$, the $\frac{1}{2}k$ edge-ends in $S_{2k-1,b}$, and the edge-end $(3k-1)_a$ in $S_{3k,b}\inv$.
\end{itemize}
This completes the case when $k\equiv 0 \mod 6$.

\mycase{$k \equiv 4 \mod 6.$} There is one vertex, given by the cyclic concatenation $$S_{3k,a} \ S_{k+1,a} \ S_{2k,a} \ S_{3k,b} \ S_{2k-1,b}\inv S_{k,b}\inv \ S_{k+1,b} \ S_{2k,b}.$$

In the case that $k=4$, the single vertex is given by the cyclic sequence 
\[
\begin{array}{cccccccccccc}
(12)_a & (2)_a & (1)_b &  (3)_a &  (2)_b &  (4)_a &  (3)_b &  (5)_a &  (6)_b\inv &  (8)_a &  
(9)_b &  (10)_b\\
+  & +  & - &  + &  + &  - &  + &  - &  + &  - &  - &  + \\ \\
 
 (11)_b & (12)_b &  (1)_a\inv &  (11)_a\inv &  (9)_a\inv &  (7)_b &  (6)_a\inv &  (4)_b\inv &  (5)_b\inv &  (7)_a &  (8)_b\inv &  (10)_a\\
-  &  +  &  -  &  -  &  -  &  +  &  +  &  -  &  -  &  +  &  -  &  +
\end{array}
\]
Here, we have indicated under each edge-end whether it belongs to an untwisted edge ($+$) or a twisted edge ($-$). Note that there is exactly one all-``$+$" consecutive subsequence of length three (on ends $(10)_a (12)_a (2)_a$), and exactly one all-``$-$" subsequence of length three (on ends $(1)_a\inv (11)_a)\inv (9)_a\inv$). Thus, any isomorphism of $(G,\ell_G)$ with its Petrial necessarily maps each of these length three subsequences to the other. Considering the surrounding edge-ends, however, even allowing for reversing the orientation, we see this is not possible. The subsequence $+-+\framebox{$- - -$}++-$ needs to map to $-+-\framebox{$+++$}--+$, whereas the subsequence actually present is $-+-\framebox{$+++$}-++$. 

We now address the cases where $k\geq 10$ by analyzing the lengths of sequences of consecutive twisted, or consecutive untwisted, edge ends. 
\begin{itemize}
    \item From $(3)_a$ to $(k-2)_a$ in $S_{3k,a}$ is a length $2k-6$ sequence of consecutive untwisted edges ends. 
    \item From $(k)_a$ in $S_{3k,a}$ to $(2k)_a$ in $S_{2k,a}$ is a sequence of edge ends alternating between twisted and untwisted which both starts and ends with twisted edges.
    \item From $(2k+1)_b$ in $S_{2,a}$ to $(3k)_b$ in $S_{3k,b}$  is a sequence of edge ends alternating between twisted and untwisted which starts with a twisted-edge end and ends with an untwisted-edge end.
    \item From $(1)_a\inv$ in $S_{3k,b}$ to $(2k+1)_a\inv$ in $S_{2k-1,b}\inv$ is a length $\frac{1}{2}k+2$ sequence of twisted-edge ends.
    \item From $(2k-1)_b$ in $S_{2k-1,b}\inv$ to $(k+2)_a\inv$ is a length $\frac{2}{3}(k-1)$  sequence of untwisted-edge ends.
    \item $(k)_b\inv $ in $S_{k,b}\inv$ is twisted, and then from $(k+1)_b\inv$ in $S_{k+1,b}$ to $(2k)_b\inv$ in $S_{2k,b}$ is a sequence of edge ends alternating between twisted and untwisted which starts and ends with a twisted-edge end.
    \item from $(2k+2)_a$ in $S_{2k,b}$ to $(3k)_a$ in $S_{3k,a}$ is a length $\frac{1}{2}k$ sequence of untwisted-edge ends. 
\end{itemize}
Since there is a length $2k-6$ sequence of consecutive untwisted edges ends but there is no sequence of consecutive twisted edge ends of that length, we see that $(G,\ell_G)$ cannot be isomorphic to its Petrial.
\end{proof}

\section{Code discussion} \label{Code discussion}

In this section we describe some of the more interesting points involved in developing a computational implementation to make use of the theoretical framework discussed above.

\subsection{Labeled ribbon graphs}

Let $(G,\ell)$ be a ribbon graph. In the implementation, where we do not have topological objects themselves but their symbolic encodings, we treat the elements of $\{1,2, \ldots,n\}$ as the names of the edges and take $\ell$ in all cases to be the identity function.

\subsection{Single-vertex ribbon graphs as chord diagrams} \label{svogs_as_chordiags}

An OEB can be represented as a signed chord diagram, with the outer circle representing the vertex, the chords representing the ribbons, and a chord being signed positive or negative according as the corresponding ribbon is untwisted or twisted. For implementation purposes, we linearize the chord diagram by making an arbitrary choice of a designated point on the circle. Thus, to work with chord diagrams as representations of OEBs, we need to consider orbits under a dihedral action; this is discussed more below.

We use a $2k$-tuple $D$ to represent a chord diagram with $k$ chords in three different ways, as given below. In all cases, a negative entry indicates the corresponding edge is twisted, and a positive entry indicates it is not twisted.

\begin{itemize}
\item If $D$ is in \emph{offset form} then  $|D(i)|=j$ indicates that the ribbon with one end at spot $i$ has its other end at spot $i+j \mbox{\ mod\ } 2k$. This representation has the benefit that a cyclic shift does not change the values, so it is useful when considering equivalence modulo a dihedral action and, more generally, isomorphism of unlabeled ribbon graphs.

\item If $D$ is in \emph{end-spot form} then $|D(i)|=j$ indicates that the ribbon with one end at spot $i$ has its other end at spot $j$. This representation has the benefit that it is easy to work with in terms of splicing and concatenating.

\item If $D$ is in \emph{end-label form} then $|D(i)|=j$ indicates that the ribbon with label $j$ has one end at spot $i$. This representation works best with the $(G,\ell)$ notation.
\end{itemize}

Conversion between forms is a quick linear process.

\subsection{Isomorphism of ribbon graphs using chord diagrams}

As noted above, equivalence of single-vertex labeled ribbon graphs represented as linearized chord diagrams is determined by dihedral action. Let $(G,\ell)$ and $(G',\ell')$ be OEB's, and let $\hat D, \hat D'$ be the chord diagrams in end-label form corresponding to  $(G,\ell)$ and $(G',\ell')$, respectively. If $\phi\colon (G,\ell) \to (G',\ell')$ is an isomorphism, then there is a corresponding dihedral permutation $\sigma \in \D_{2k}$ such that the  bijection on edges corresponding to $\phi$ is $\ell'(\hat D '(\sigma(i))) \mapsto \ell(\hat D(i))$. Since we take $\ell$ and $\ell'$ to be the identity map and $\phi$ preserves labels, we have $\hat D'(\sigma(i)) = \hat D(i)$ for $i = 1, 2, \ldots, 2k$. Of course, this framework accounts for the case when $\hat D$ and $\hat D'$ are two different chord diagrams in end-label form for the same OEB $(G,\ell)$.

Given $(G,\ell)$ and $(G',\ell')$ and an isomorphism $\phi\colon G \to G'$ of unlabeled graphs, there is a permutation $\pi\in S_k$ such that $\phi$ gives an isomorphism $(G,\ell) \to (G',\ell'\pi^{-1})$.

The necessity of the permutation $\pi$ when working with the end-label form of linearized chord diagrams slows computation down considerably. Offset form is preferable for this purpose, but working with the dihedral action in this case requires slightly more care. Let $\D_{2k} = \langle \rho, \tau \rangle$ where $\rho$ represents the order-$k$ cyclic shift and $\tau$ represents the order-2 flip. For any size-$2k$ linearized chord diagram $D$ in offset form we write $\underline{D}$ to denote the size-$2k$ linearized chord diagram given by $\underline{D}(i) = 2k - D(i)$. We define a right-action of $\D_{2k}$ as follows: If $\sigma \in \langle \rho \rangle < \D_{2k}$ then $D\cdot\sigma = D\circ \sigma$, whereas if $\sigma \in \tau\langle \rho \rangle$ then $D\cdot\sigma = \underline{D}\circ\sigma$. Since $\langle \rho \rangle$ is an index-2 (hence normal) subgroup and $\underline{\underline{D}}=D$, this action is well defined.

With the dihedral action set up this way, isomorphism of unlabeled graphs becomes computationally much easier. If $D$ and $D'$ are the linearized chord diagrams in offset form corresponding to $(G,\ell)$ and $(G',\ell')$,  respectively, then $G$ is isomorphic to $G'$ as an unlabeled graph if and only if there is $\sigma \in \D_{2k}$ such that $ D' \cdot \sigma = D$. Given such a $\sigma$, we can recover the corresponding permutation $\pi$ as the unique permutation satisfying $\hat D'(\sigma(i)) = \pi\hat D(i)$ for all $i= 1,2, \ldots, 2k$, where $\hat D', \hat D$ are the linearized chord diagrams in end-label form corresponding to  $(G,\ell), (G',\ell')$,  respectively.

\subsection{Enumerating chord diagrams}

To enumerate chord diagrams up to isomorphism, we first used the algorithm of Nijenhuis and Wilf \cite{NW} to enumerate linear diagrams. The basic approach of the algorithm is to build up larger diagrams from smaller ones; accordingly, the end-spot representation proved to be most useful for these computations. In order to obtain the desired list of representatives of isomorphism classes of [cyclic] chord diagrams, after reinterpreting the linear diagrams as cyclic they were put into a canonical linear representation modulo the dihedral action, and then duplicates were removed from the resulting list.

\subsection{Ribbon operations}

To perform ribbon operations, we encode ribbon graphs using {\it jewels}, which are properly 4-colored 4-regular simple graph with colors red, green, blue, and yellow, say, in which the red-green-blue subgraph is a disjoint union of 4-cliques. A jewel encodes an embedding of a graph $G$ as follows:
\begin{itemize}
    \item[] vertices of $G$ = components of red-yellow subgraph
    \item[] edges of $G$ = components of red-blue subgraph
    \item[] faces of $G$ = components of yellow-blue subgraph
\end{itemize}
Jewels are a generalization of {\it gems}, which incorporate only red, blue, and yellow edges. Gems and jewels were used by Lins \cite{Lins} to reframe the work of Wilson \cite{Wil79} on which this paper is based. Given a jewel $\Gamma$, let $G = G(\Gamma)$ denote the corresponding graph. For an edge $e$ of 
$G(\Gamma)$, let $K_e$ denote the red-green-blue 4-clique corresponding to $e$. Ribbon graph operations may now be realized by the following actions:
\begin{itemize}
    \item[] dual of an edge $e$: switch the red and blue colorings in $K_e$
    \item[] Petrial of an edge $e$: switch the blue and green colorings in $K_e$
\end{itemize}

A jewel which encodes a single vertex graph has a single red-yellow jewel-cycle $v$. To convert this to a linearized chord diagram requires choosing a directed red jewel-edge to be the initial jewel-edge and ordering the other red jewel-edges sequentially around $v$. 

\subsection{Finding stabilizers}
Given labeled ribbon graph $(G,\ell)$ with $k$ ribbons, we want all pairs $(\bga,\pi)$ such that $(\bga,\pi)(G,\ell) \cong (G,\ell)$. Let $D, \hat D$ be linearized chord diagrams for $(G,\ell)$ in offset form and label form, respectively. For each $\bga \in \fS^n$, apply $\bga$ to $(G,\ell)$ and determine the corresponding linearized chord diagrams $D_{\bga}, \hat{D}_{\bga}$ in offset and label forms, respectively. For each $\sigma \in \D_{2k}$ such that $D_{\bga\sigma} \equiv D$, define $\pi\in S_k$ by $\pi^{-1}\hat{D}_{\bga\sigma} \equiv \hat D$, and return $(\bga\pi^{-1},\pi) = (1,\pi)(\bga,\iota)$. 

\subsection{Finding self-twual graphs}

Suppose we are given $n$-ribbon $(H,\ell_H)$, ribbon-group element $\bga$ and permutation $\mu \in S_n$ such that $(\bga,\mu)(H,\ell_H) \cong (H,\ell_H)$, and $\bga' \in \fS^n$. We want a graph $(G,\ell_G)$ of the form $(G,\ell_G) = (\bal,\iota)(H,\ell_H)$, for some $\bal\in\fS^n$, which is self-$\bga'$ by $\mu$. By Theorem \ref{natural prop},  $\bal$ must satisfy $\bal\bga\phi_\mu(\bal^{-1}) = \bga'$. Since $\phi_\mu$ is a homomorphism, this gives $\phi_\mu(\bal) = (\bga')^{-1}\bal\bga$ which, in coordinates, says
\begin{equation}\label{eq:alpharelation}
(\alpha_{\mu^{-1}(1)}, \alpha_{\mu^{-1}(2)}, \ldots, \alpha_{\mu^{-1}(n)}) \ = \
\left((\gamma'_1)^{-1}\alpha_1\gamma_1, (\gamma'_2)^{-1}\alpha_2\gamma_2, \ldots, (\gamma'_n)^{-1}\alpha_n\gamma_n \right)
\end{equation}

We use the following algorithm to find all possible $\bal$: Consider each cycle $$C = \left(i \; \mu^{-1}(i) \; \mu^{-2}(i) \; \ldots \; \mu^{-r}(i)\right)$$ in the cycle decomposition of $\mu$, and for each $C$ consider each element $a \in \fS$. Assign $a$ to $\alpha_i$ and use Equation (\ref{eq:alpharelation}) to iteratively determine $\alpha_{\mu^{-1}(i)}, \alpha_{\mu^{-2}(i)}, \ldots, \alpha_{\mu^{-r}(i)}.$ If the determined value of $\alpha_{\mu^{-r}(i)}$ agrees with $a$ then we record the computed data as an option for the portion of $\bal$ corresponding to cycle $C$. If each cycle of $\mu$ had at least one computation recorded, assemble all possible $\bal$ by choosing, in all possible ways, one of the available options for each portion of $\bal$.

\section{Further directions}\label{sec:further directions}

The work outlined here would be greatly facilitated by a systematic way to study OEBs.  In particular, a canonical choice of OEB representative for each orbit would streamline not only computer searches, but the theory of twuality in general.

Also, while we found many examples of Class III graphs, none of the examples given here are canonically self-trial. This raises the question of whether or not canonically self-trial graphs exist.

Moreover, the results of Section \ref{propagation} have broader implications than just for the OEBs we used in the examples given here.  In particular, the results apply to any $H$, not just an OEB.  Thus, it is possible to take any self-twual graph $H$ and use conjugation to generate others. Numerous examples of various kinds of self-twual graphs are known.  It is already known and clear that conjugating by the same group element throughout will yield another self-twual graph.  However, the conjugacy table shows that at each edge there is a choice of several elements that can be applied.  This opens the potential for many more self-twual graphs arising from any known example.  It is then a matter of checking whether the result is isomorphic to the known example, which is easy in the canonical case. Given the wealth of existing information in the literature about regular maps and their automorphism groups, Theorem \ref{automorph duals} should lead to a rich source of new graphs with desirable self-twuality properties of all kinds.

Another natural direction following from this research is generalizing it to delta-matroids, which are to embedded graphs as matroids are to abstract graphs.  There are recently developed frameworks extending twuality operations of partial duality and partial petrie duals to delta-matroids, and the ribbon group action lifted to this setting would open investigation into various forms of self-twuality for delta-matroids. See \cite{CMNR19a, CMNR19b}.

\bibliographystyle{amsplain}

\end{document}